\documentclass[final,3p]{CSP}
\usepackage{amssymb}
\usepackage{changepage}

\usepackage{amsmath}
\usepackage{graphicx}
\usepackage[colorinlistoftodos]{todonotes}
\usepackage[colorlinks=true, allcolors=blue]{hyperref}
\usepackage{enumitem}

\usepackage{amsmath,verbatim,amssymb,amsfonts,amscd, graphicx}
\usepackage{mathabx}
\usepackage{breqn}
\usepackage{float}
\usepackage{graphics}
\usepackage{stmaryrd} 
\usepackage{textcomp}
\usepackage{setspace}
\usepackage{algorithm}
\usepackage{textgreek}
\usepackage{algpseudocode}
\usepackage{amsthm}
\usepackage{enumitem}   
\usepackage{titlesec}
\usepackage{caption, subcaption}
\usepackage[toc]{appendix}

\usepackage{ragged2e}

\titleformat{\paragraph}
{\normalfont\normalsize\bfseries}{\theparagraph}{1em}{}
\titlespacing*{\paragraph}
{0pt}{3.25ex plus 1ex minus .2ex}{1.5ex plus .2ex}

\theoremstyle{plain}

\newtheorem{proposition}{Proposition}[section]

\newtheorem{assumption}{Assumption}[section]
\newtheorem{theorem}[proposition]{Theorem}
\newtheorem{lemma}[proposition]{Lemma}
\theoremstyle{remark}
\newtheorem{remark}{Remark}[section]

\title{  }

\begin{document}

\begin{frontmatter}

\title{A unifying vision of Particle Filtering and Explicit dual Model Predictive Control}

\author[UNIMELB]{Emilien  Flayac}\ead{emilien.flayac@unimelb.edu.au}    
\author[ONERA]{Karim Dahia}\ead{karim.dahia@onera.fr}               
\author[ONERA]{Bruno H\'eriss\'e}\ead{bruno.herisse@onera.fr}   
\author[ENSTA]{Fr\'ed\'eric Jean}\ead{frederic.jean@ensta-paris.fr}

\address[UNIMELB]{Electrical and Electronical Engineering Department, University of Melbourne, Parkville VIC 3010, Melbourne, Australia}

\address[ONERA]{ONERA, Palaiseau, France}  
\address[ENSTA]{UMA, ENSTA Paris, Institut Polytechnique de Paris, Palaiseau, France} 


\begin{abstract}\rm
\begin{adjustwidth}{2cm}{2cm}{\itshape\textbf{Abstract:}} 

\begin{justify}

This paper presents a joint optimisation framework for optimal estimation and stochastic optimal control with imperfect information. It provides a estimation and control scheme that can be decomposed into a classical optimal estimation step and an optimal control step where a new term coming from optimal estimation is added to the cost. It is shown that a specific particle filter algorithm allows one to solve the first step approximately in the case of Mean Square Error minimisation and under suitable assumptions on the model. Then, it is shown that the estimation-based control step can justify formally the use of Explicit dual controllers which are most of the time derived from empirical matters. Finally, a relevant example from Aerospace engineering is presented.

\end{justify}
\end{adjustwidth}
\end{abstract}
\end{frontmatter}
\section{Introduction}
Optimal estimation and control problems arise when one wants to reconstruct and monitor in an optimal way the state of a system using only partial information. Such problems are widespread, for instance, in chemical engineering \cite{kumar2012model}, in electrical and mechanical engineering \cite{bolton2003mechatronics}, in mathematical finance \cite{mamon2007hidden} and in aerospace engineering \cite{eren_model_2017}. In the presence of nonlinearities and disturbances, these problems are challenging for mainly two reasons. 

First, nonlinear estimation (or filtering) problems are known to be difficult on their own. While Kalman filters or optimisation-based estimators are cheap and perform well in the case of uni-modal uncertainty \cite{ristic2004beyond,rawlings2006particle}, they become inefficient when the conditional distribution is spread over several modes. A typical solution is to use particle filters, which are able to find these modes. As particle filters are Monte Carlo approximations of the optimal filter, there exists a zoo of theoretical convergence results. For example, in \cite{crisan_survey_2002}, almost sure and L2-convergence of particle filters are reviewed.  In \cite{del_moral_central_1999} and \cite{del_moral_monte-carlo_2002}, the authors show a Central Limit Theorem for classes of particle filters. In practice though, it is very common to assess the performance of a filter using the Mean Square Error (MSE). Intuitively, one would expect that, if the number of particles is sufficiently high, the MSE generated by a particle filter converges to the optimal MSE associated to the optimal filter. The MSE can also be a tool for assessing the stability a filter, see \cite{reif_stochastic_1999} and \cite{karvonen_stability_2014} for stability results of nonlinear Kalman filters. Consequently, showing that a particle filter reaches a quasi-optimal MSE might be a way to show its moment stability provided that another suboptimal filter is stable. Regrettably, the MSE involves integrals of unbounded functions w.r.t. the particle filter that are not handled by classical convergence results.   Thus, even though MSE's relevance can be discussed in multi-modal cases, MSE convergence results for particle filters lack in the literature. Besides, despite the very good performance of particle filtering in very nonlinear applications, only a few papers deals with including them in optimal control schemes, see for instance \cite{bayard_implicit_2008,sehr_particle_2016,flayac_nonlinear_2017,flayac_dual_2018}.     

Secondly, stochastic optimal control problems with imperfect information are more difficult than their full information counterpart. Actually, when only partial information is available, optimal controls have two roles. They must guide the system in a standard way and probe information to be aware of and improve the quality of the future information. It is called the \emph{dual effect} property of the control \cite{bar-shalom_dual_1974}. As optimal solutions are usually intractable, suboptimal control laws (also called dual controllers) are designed instead with the requirement to keep the property of dual effect. There are two main types of suboptimal dual controllers: implicit ones where one tries to approximate the Bellman equation usually by preserving the feedback structure on the information inside the optimal control problem (see \cite{bayard_implicit_2008,hanssen_scenario_2015,subramanian_non-conservative_2016}) and  explicit ones where an external excitation is introduced in the system to make it actively look for more information. The excitation can take the form of a constraint on the future information \cite{telen_study_2017} or of an additional term in the cost representing a loss of information. The latter way is called integrated experiment design. See \cite{mesbah_stochastic_2017} for a review in the Stochastic Model Predictive Control framework and \cite{unbehauen2000adaptive} for a survey in Adaptive Control. Information constraints may lead to infeasibility issues and are not very flexible. For this reason, in this article, we focus on Explicit dual control methods by integrated experiment design. Implicit methods are solidly formally grounded because they try to reproduce Bellman equation's behaviour. However, explicit ones are less well justified. In fact, the general link between the original optimal control problem and the new one is not clear since the modifications usually come from empirical considerations. The need for an external excitation that makes the control actively learn is generally assumed and justified in specific cases only.  

In other words, the dual effect property means that estimation and control must be coupled in their design. In particular, the separation principle cannot be applied for general nonlinear systems. Several attempts have been made to study coupled estimation and control in a general framework. In \cite{andrieu_unifying_2009}, a general formalism for joint nonlinear observer and control design is presented in a continuous-time deterministic framework. In \cite{copp_nonlinear_2014}, a min-max formalism for combined Moving Horizon Estimation and Model Predictive Control is presented but without imposing the dual effect property. To the best of our knowledge, there exists no joint formulation of the problem of optimal estimation and control in a general discrete time stochastic framework.

In this paper, an infinite-horizon multistage stochastic optimisation problem that gathers an optimal estimation problem and a stochastic optimal control problem  with imperfect information is proposed. After writing its Bellman equation, one can decompose the problem into two steps. The first step is a classical optimal estimation step and the second one is a modified stochastic optimal control step in which the optimal estimation error is added to the cost. These steps allow us to justify the use of particle filtering and Explicit dual control in an estimation/control scheme. In fact, we prove the near-optimality of the empirical mean of a specific particle filter in the case of MSE minimisation, with a rate of convergence. This means that the first step can be solved approximately by a particle filter under suitable assumptions on the model. Afterwards, we make a strong analogy between the modified control problem and integrated experiment design. The idea is that the additional empirical cost could be seen as a approximation of the optimal estimation error as it is itself a measure of information. Explicit dual control is then an practical solution of our second step.
Finally, we present an example coming from Aerospace engineering in which both particle filtering and Explicit dual control are very relevant. We also check that this application satisfies the assumptions of our near-optimality results.

The paper is structured as follows. Section \ref{sec_notations} gathers important notations used in the sequel. Section \ref{sec_elements_esti_cont} recalls some basics of optimal estimation and stochastic optimal control with imperfect information. Section \ref{sec_coupled_estimation_control} presents the coupling multistage program along with its analysis and re-decomposition into two steps. Section \ref{sec_particle_near_opti} contains the main convergence results concerning particle filtering and MSE minimisation. Section \ref{sec_modified_explicit} presents the link between the modified control problem and Explicit dual control. Finally, Section \ref{sec_example} describes how the example of Terrain-Aided Navigation fits the proposed framework.
\section{Notations}\label{sec_notations}
Let $(\Omega, \mathcal{F}, P)$ be a probability space. In the following, random variables refer to $\mathcal{F}$-measurable functions defined on $\Omega$. For $i\in \mathbb{N}^*$, $\mathcal{B}(\mathbb{R}^i)$ denotes the set of Borel sets of $\mathbb{R}^i$ and  $\mathcal{P}(\mathbb{R}^i)$ the set of probability distributions on $\mathbb{R}^i$. For a random variable $X$ and a probability distribution, $X\sim p$ means that $p$ is the probability law of $X$. $P(\cdot\vert\cdot)$ and $E(\cdot\vert\cdot)$ denotes the conditional probability and expectation. For $X$ and $Y$ two random variables valued respectively in $\mathbb{R}^i$ and $\mathbb{R}^j$ and $A \in \mathcal{B}(\mathbb{R}^j)$  , $P(Y\in A\vert X=x)$ is uniquely defined only for almost all $x$ in $\mathbb{R}^i$ considering the distribution of $X$. However, we will omit it several times in this paper when it is not important. In the sequel, a.a. is an abbreviation for almost all. For $\mu \in \mathcal{P}(\mathbb{R}^i)$, $f:\mathbb{R}^i \rightarrow \mathbb{R}^j $ $\mu$-integrable, and $A \in \mathcal{B}(\mathbb{R}^i)$, we denote the integral of $f$ w.r.t $\mu$ on $A$ by  $\int_A f(x)\mu(\mathrm{d}x)$. Integrals w.r.t the Lebesgue measure are denoted by $\int_A f(x)\mathrm{d}x$. When conciseness is required, we use a bracket notation for the integrals on the whole space, so that $\int_{\mathbb{R}^i} f(x)\mu(\mathrm{d}x)=\langle \mu,f\rangle$. For $x\in \mathbb{R}^i$, $\delta_x$ denotes the Dirac probability measure centered at $x$. Id stands for the identity fonctions on $ \mathbb{R}^i$. In the sequel, all the optimisation problems are assumed to have a solution, in particular the notation 'min' is used instead of 'inf'.
\section{Elements of stochastic estimation and control}\label{sec_elements_esti_cont}
\subsection{Setup}
We consider a discrete-time process $X={(X_k)}_{k \in \mathbb{N} }$ valued in $\mathbb{R}^{n_x}$ representing the state of a controlled stochastic dynamical system  described by the following equation:
\begin{align}
X_{k+1}&=f(X_{k},U_{k},{\xi}_{k}) \label{dyn_sto_gen}, \; \forall k \in \mathbb{N}\\
X_0&\sim p_0,\notag
\end{align}
where:
\begin{itemize}
\item $p_0$ is a probability law on $\mathbb{R}^{n_x}$;
\item 
${(U_k)}_{k \in \mathbb{N} }$ is the control process valued in $\mathcal{U}\subset\mathbb{R}^{n_u}$. $\mathcal{U}$ is the set of admissible control values;
\item ${(\xi_k)}_{k \in \mathbb{N}}$ are i.i.d. random variables valued in $\mathbb{R}^{n_{\xi}}$ distributed according to $p_\xi$. For each $ k \in \mathbb{N}$, $\xi_k$ represents an external disturbance on the dynamics;
\item $f$: $\mathbb{R}^{n_x}\times\mathbb{R}^{n_u}\times\mathbb{R}^{n_{\xi}} \longrightarrow \mathbb{R}^{n_x}$ is measurable.
\end{itemize}
In fact, equation \eqref{dyn_sto_gen} defines ${(X_k)}_{k \in \mathbb{N} }$ as a Markov Decision Process on $\mathbb{R}^{n_x}$. Its transition kernel, denoted by $K$, is assumed to has a density with respect to the Lebesgue measure meaning that for all $ A \in \mathcal{B}(\mathbb{R}^{n_x})$, for $x\in \mathbb{R}^{n_x}$ and  $u\in \mathbb{R}^{n_u}$,
\begin{align}
    P(X_{k+1}\in A \vert X_k=x,\; U_k=u)&= \int_A K(z,x,u)\mathrm{d}z\label{cond_proba_dyn},
\end{align}
where $K$: $\mathbb{R}^{n_x}\times\mathbb{R}^{n_x}\times\mathbb{R}^{n_u} \longrightarrow \mathbb{R}^+$ is measurable and satisfies for $x\in \mathbb{R}^{n_x}$ and  $u\in \mathbb{R}^{n_u}$, $\int_{\mathbb{R}^{n_x}} K(z,x,u)\mathrm{d}z=1$. Additionally, we assume that the state of the system is only available through some observations represented by a stochastic process ${Y=(Y_k)}_{k \in \mathbb{N}}$ valued in $\mathbb{R}^{n_y}$ which verifies for any $k \in \mathbb{N} $,
\begin{align}
Y_k=h(X_k,\eta_k)\label{obs_sto_gen},
\end{align}
where:
\begin{itemize}
\item ${(\eta_k)}_{k \in \mathbb{N} }$ are i.i.d. random variables valued in $\mathbb{R}^{n_{\eta}}$ distributed according to $p_\eta$; for each $ k \in \mathbb{N}$, $\eta_k$  represents an external disturbance on the observations. 
\item $h : \mathbb{R}^{n_x}\times\mathbb{R}^{n_{\eta}} \longrightarrow \mathbb{R}^{n_y}$ is measurable.
\end{itemize}

In the following, we assume that the conditional distribution defined by equation (\ref{obs_sto_gen}) has a density with respect to the Lebesgue measure such that there exists a likelihood function $\rho: \mathbb{R}^{n_y}\times \mathbb{R}^{n_x}\rightarrow \mathbb{R}^+$.    
Therefore, for $k \in \mathbb{N}$, for $x$ in $\mathbb{R}^{n_x}$ for $B \in \mathcal{B}(\mathbb{R}^{n_y})$ :
\begin{align*}
    P(Y_{k}\in B \vert X_k=x_k)&=\int_B \rho(y,x)\mathrm{d}y,
\end{align*}
with $\rho$ measurable and  $\int_{\mathbb{R}^{n_y}} \rho(y,x)\mathrm{d}y=1$. For $k \in \mathbb{N}$, we define the vector of \textit{available information} $I_k$ as follows:
\begin{align}
I_k&=(Y_0,U_0,\dots,Y_{k-1},U_{k-1},Y_k)\label{inf_vec}.
\end{align}

Note that $I_k$ represents all the \textit{values} that are available to compute an estimator of the state and a control. Another important quantity in this framework is the conditional distribution of $X_k$ given $I_k$, called the filtering distribution (or optimal filter) and denoted by $\mu_k$. It is of central importance in Bayesian filtering as it contains and weighs the possible values of the current state $X_k$ given only the value of $I_k$. It is known that  $\mu_k$ satisfies the nonlinear filtering equations which can be summed up as follows:
\begin{align}
    \mu_{k+1}&=F\left(\mu_k,Y_{k+1},U_k\right),\label{filtering_equation_red}
\end{align}
where $F$ : $\mathcal{P}(\mathbb{R}^{n_x})\times\mathbb{R}^{n_y}\times\mathbb{R}^{n_{u}} \longrightarrow \mathcal{P}(\mathbb{R}^{n_x})$ and $\mu_0$ is supposed to be known.


 The problem of estimation and control treated in this paper can be formulated as finding an estimator of the state, $\widehat{X}_k$ and a control $U_k$ as functions of $I_k$ such that:
\begin{align}
 \widehat{X}_k=\pi_k^e(I_k), \qquad
 U_k&=\pi_k^c(I_k),\label{control_policy_info_vec_imperf} 
\end{align}
where $\pi_k^e$: ${(\mathbb{R}^{n_y}\times\mathbb{R}^{n_u})}^{k}\times \mathbb{R}^{n_y} \longrightarrow \mathbb{R}^{n_x}$ and $\pi_k^c$: ${(\mathbb{R}^{n_y}\times\mathbb{R}^{n_u})}^{k}\times \mathbb{R}^{n_y} \longrightarrow \mathcal{U}\subset\mathbb{R}^{n_u}$ are measurable. Sequences of the form $\pi^e=(\pi_0^e,\dots,\pi_k^e,\dots)$ (resp. $\pi^c=(\pi_0^c,\dots,\pi_k^c,\dots)$) are called estimation (resp. control) policies.

Besides, it can be shown that $\mu_k$ carries as much information as $I_k$. More precisely, looking at $\mu_k$ as a random variable on $\mathcal{P}(\mathbb{R}^{n_x})$ equipped with the Borel $\sigma$-algebra for the weak topology, it is a sufficient statistics (see \cite{bertsekas_dynamic_2011,bertsekas_stochastic_2004}). It means that the estimation and control policies can also be looked for as functions of $\mu_k$ instead of $I_k$. Since equation \eqref{filtering_equation_red}  describes a time homogeneous Markov Chain on $\mathcal{P}(\mathbb{R}^{n_x})$ \cite{bertsekas_stochastic_2004,stettner_invariant_1989}, one can look for time-homogeneous functions such that, for $k\geq 0$:
\begin{align}
\widehat{X}_k=\pi_0^e\left(\mu_k\right),\qquad
 U_k=\pi_0^c\left(\mu_k\right),\label{control_policy_info_vec_imperf_inf_state_filter_time_homo}   \end{align}
where  $\pi_0^e$: $\mathcal{P}(\mathbb{R}^{n_x}) \longrightarrow \mathbb{R}^{n_x}$ and $\pi_0^c$: $\mathcal{P}(\mathbb{R}^{n_x}) \longrightarrow \mathcal{U}\subset\mathbb{R}^{n_u}$ are also measurable with the policies being $\pi^e=(\pi_0^e,\pi_0^e,\dots)$ and $\pi^c=(\pi_0^c,\pi_0^c,\dots)$.

The idea of the following is to look for $\pi^e$ and $\pi^c$ in a optimal way. With this in mind, the basics of optimal estimation and stochastic optimal control with imperfect information are recalled.
\subsection{Optimal estimation}\label{sec_optimal_estimation}

Classically, optimal estimation is concerned with finding an estimator of $X_k$ as a function of $I_k$ that minimises in average a general measure of the estimation error denoted by $g^e$. This leads to the following optimisation problem: 
\begin{align}
\begin{array}{rrclcc}
\displaystyle \underset{\pi^e_k}{\text{min}} & \multicolumn{3}{l}{E\left[  g^e(X_k,\widehat{X}_k)|I_k\right]} \\
\textrm{s.t.} & \widehat{X}_k  =  \pi^e_k({I}_k).
\end{array}
\label{optimal_estimation_gen_inf_vec}
\end{align}

Actually, the conditional expectation in the cost function from Problem \eqref{optimal_estimation_gen_inf_vec} can also be represented as an integral of $g^e$ w.r.t. the filtering distribution  $\mu_k$ and a new cost $\tilde{g}_e$ can be written as follows for any $\mu \in \mathcal{P}(\mathbb{R}^{n_x})$ and $\hat{x}\in \mathbb{R}^{n_x} $:
\begin{align*}
    \tilde{g}^e(\mu,\hat{x})&=\langle\mu,g^e(\cdot,\hat{x})\rangle,
\end{align*}
where $\langle\cdot,\cdot \rangle$ denotes the integral operator. Since $g^e$ does not depend explicitly on time, one obtains the following reformulation of Problem \eqref{optimal_estimation_gen_inf_vec} using a time homogeneous estimation policy of the form of  \eqref{control_policy_info_vec_imperf_inf_state_filter_time_homo}:
\begin{align}
\begin{array}{rrclcc}
\displaystyle \underset{\pi^e_k}{\text{min}} & \multicolumn{3}{l}{\tilde{g}^e(\mu_k,\widehat{X}_k)} \\
\textrm{s.t.} & \widehat{X}_k  =  \pi^e_0(\mu_k).
\end{array}
\label{optimal_estimation_gen_state_filter}
\end{align}

Problem \eqref{optimal_estimation_gen_inf_vec} and \eqref{optimal_estimation_gen_state_filter} have generally no analytical solutions except in a few cases including the case where $g^e(x,\hat{x})={\Vert\hat{x}-x\Vert}^2$ with ${\Vert \cdot \Vert}$ standing for the Euclidean norm on  $\mathbb{R}^{n_x}$.
 The latter problem is referred to as the  conditional Mean Square Error (MSE) minimisation problem and reads:
 
 \begin{align}
\begin{array}{rrclcc}
\displaystyle \underset{\pi^e_k}{\text{min}} & \multicolumn{3}{l}{E\left[  {\Vert \widehat{X}_k-{X}_k\Vert}^2|I_k\right]} \\
\textrm{s.t.} & \widehat{X}_k  =  \pi^e_k({I}_k).
\end{array}
\label{optimal_estimation_MSE_cond}
\end{align}

By simple calculations, one gets that the almost surely optimal estimator  w.r.t. the distribution of  $I_k$ is the  expectation of $X_k$ conditionally to $I_k$, defined in the following by $X_k^*:=E[X_k\vert I_k]$. Note that  $\widehat{X}_k^*$  can be rewritten as the integral of {Id} w.r.t. $\mu_k$ such that:
\begin{align}
     \widehat{X}_k^*=\langle\mu_k,\mathrm{Id} \rangle.
     \label{conditional_expectation_integral}
\end{align}

It is well known from \cite{anderson_optimal_1979} that $\widehat{X}_k^*$ is also the solution of the total MSE minimisation problem which reads:
\begin{align}
\begin{array}{rrclcc}
\displaystyle \underset{\pi^e}{\text{min}} & \multicolumn{3}{l}{E\left[  {\Vert\widehat{X}_k-X_k\Vert}^2\right]} \\
\textrm{s.t.} & \widehat{X}_k  =  \pi^e_k({I}_k).
\end{array}\label{optimal_estimation_total_MSE}
\end{align}

\subsection{Stochastic Optimal Control with imperfect information}\label{sec_optimal_control}

We consider classical time-homogeneous, infinite horizon stochastic optimal control problems with imperfect information. To do so, we define
 a time homogeneous instantaneous cost $g^c$: $\mathbb{R}^{n_x}\times\mathbb{R}^{n_u}\times\mathbb{R}^{n_{\xi}} \longrightarrow \mathbb{R}^+$ and a discount factor $\alpha \in )0,1]$. Using also the dynamics \eqref{dyn_sto_gen}, the observation equation \eqref{cond_proba_dyn} and the information dynamics \eqref{inf_vec}, one gets the following problem in the information vector space, for any $i_0$:
 
 \begin{equation}
\begin{array}{rrclcc}
V_0(i_0)=\displaystyle \underset{\pi_{k}^c}{\text{min}} & \multicolumn{3}{l}{E\left[\sum_{k=0}^{+\infty} \alpha^k g^c(X_k,U_k,\xi_k)| I_0 = i_0\right]} \\
\textrm{s.t.} & X_{k+1} & = & f(X_{k},U_{k},{\xi}_{k}), \\
& Y_k & = & h(X_k,\eta_k),\\
&I_{k+1}&=&(I_k,U_{k},Y_{k+1}),\\
& U_k & = & \pi_k^c(I_k), \; \forall k\geq0.
\end{array}
\label{sto_opt_control_imperfect_inf_infinite_horizon}
\end{equation}

Using the Markov structure of ${(\mu_k)}_{k\geq0}$ mentioned previously, one can reformulate Problem \eqref{sto_opt_control_imperfect_inf_infinite_horizon} as a perfect information one on $\mathcal{P}(\mathbb{R}^{n_x})$ with the new state being $\mu_k$. Additionally, as $g^c$ does not depend explicitly on time, similarly to $g^e$, one can use policies of the form \eqref{control_policy_info_vec_imperf_inf_state_filter_time_homo}. The resulting problem can be written compactly as follows,  for any $\mu \in \mathcal{P}(\mathbb{R}^{n_x})$:
\begin{equation}
\begin{array}{rrclcc}
V(\mu)=\displaystyle \underset{\pi_{0}^c}{\text{min}} & \multicolumn{3}{l}{E\left[\sum_{k=0}^{+\infty} \alpha^k\tilde{g}^c(\mu_k,U_k)| \mu_0 = \mu\right]} \\
\textrm{s.t.} &\mu_{k+1}&=&F\left(\mu_k,Y_{k+1},U_k\right),\\
& U_k & = & \pi_0^c(\mu_k), \; \forall k\geq 0,
\end{array}
\label{sto_opt_control_imperfect_inf_infinite_horizon_state_filter}
\end{equation}
with $\tilde{g}^c(\mu_k,u)=\langle\mu_k,c(\cdot,u) \rangle$ where $c$ depends on $g^c$ and on the conditional distribution of $\xi_k$ knowing $X_k$.
As Problem \eqref{sto_opt_control_imperfect_inf_infinite_horizon_state_filter} is a perfect information one, the Dynamics Programming (DP) Principle can be applied (see  \cite{bertsekas_stochastic_2004} for the details) and one obtains the following Bellman equation,  for any $\mu \in \mathcal{P}(\mathbb{R}^{n_x})$:
\begin{align}
\begin{array}{rrclcc}
V(\mu)=\displaystyle \underset{u\in \mathcal{U}}{\text{min}} & \multicolumn{3}{l}{E\left[ \tilde{g}^c(\mu,u) + \alpha V(\mu_{\ell+1})| \mu_\ell = \mu\right]} \\
\textrm{s.t.} &\mu_{\ell+1}&=&F\left(\mu_\ell,Y_{\ell+1},u\right).
\end{array}
\label{sto_dp_imperfect_inf_infinite_horizon_state_filter}
\end{align}

It can be seen from the DP Principle  \eqref{sto_dp_imperfect_inf_infinite_horizon_state_filter} that any optimal policy, denoted by $\pi^c_*$, exhibits implicit dual effect. It means that the controls influence the future available information and thus future state estimation. The term 'implicit' refers to the fact that, in the case of optimal policies, the dual effect comes from optimality and not from an external excitation. In this sense, control and state estimation cannot be separated and must be coupled in their design. In the sequel, we formalise this idea in a joint optimisation framework. 

\section{Coupled optimal estimation and control}\label{sec_coupled_estimation_control}

\subsection{Formalisation of the optimisation problem}

In this section, we consider the system \eqref{dyn_sto_gen} with the observation equation \eqref{obs_sto_gen} and focus on the infinite horizon case. The idea is to add an estimator as a variable in \eqref{sto_opt_control_imperfect_inf_infinite_horizon} in order to mix Problem \eqref{optimal_estimation_gen_inf_vec} and \eqref{sto_opt_control_imperfect_inf_infinite_horizon}. To do so we consider an augmented control $W_k=(U_k,\widehat{X}_k)$ and the corresponding augmented information vector $\tilde{I}_k$ defined recursively as follows:
\begin{align}
    \tilde{I}_{0}=Y_0,\qquad
    \tilde{I}_{k+1}=(\tilde{I}_k,W_{k},Y_{k+1}).\label{dyn_info_augmented}
\end{align}
In fact, $W_k$ is chosen as a function of $\tilde{I}_k$ as in equation \eqref{control_policy_info_vec_imperf} such that for $k\geq0$:
\begin{align*}
  U_k =\pi_k^c(\tilde{I}_k), \qquad
 \widehat{X}_k =\pi_k^e(\tilde{I}_k),\qquad
 \pi_k^{aug}=(\pi_k^c,\pi_k^e),
 \end{align*}
 \begin{align*}
 W_k=(U_k,\widehat{X}_k)=\pi_0^{aug}(I_k).
\end{align*}
 We also define the new augmented cost function $g^{aug}$ and dynamics $f^{aug}$ in the following way, for $x\in \mathbb{R}^{n_x}$,  $w=(u,\hat{x})\in \mathcal{U}\times\mathbb{R}^{n_x}$, $\xi\in \mathbb{R}^{n_\xi}$:
\begin{align}
    g^{aug}(x,w,\xi)&=g^c(x,u,\xi)+g^e(x,\hat{x}),\label{cost_control+estimation}\\
    f^{aug}(x,w,\xi)&=f(x,u,\xi),\label{dyn_sto_augmented}
\end{align}
where $g^e$: $\mathbb{R}^{n_x}\times\mathbb{R}^{n_x}\longrightarrow \mathbb{R}^+$ is a measure of the estimation error as in Problem \eqref{optimal_estimation_gen_inf_vec}
 and $g^c$: $\mathbb{R}^{n_x}\times\mathbb{R}^{n_u}\times\mathbb{R}^{n_\xi}\longrightarrow \mathbb{R}^+$ is the cost function of a classical stochastic optimal control of the form \eqref{sto_opt_control_imperfect_inf_infinite_horizon}.
Note that the augmented dynamics $f^{aug}$ does not depend on $\hat{x}$.
The main assumption in \eqref{cost_control+estimation} is that the total cost $g^{aug}$ can be separated as a sum of a control-oriented term and an estimation-oriented term. It is a mild assumption as one can consider that the true underlying estimation and control problem would be a bi-objective one with $g^c$ and $g^e$ being the two cost functions. Intuitively, $g^c$ and $g^e$ are often anti-correlated because one often needs to trade some control performance for a better estimation. See \cite{unbehauen2000adaptive} for an example in adaptive control.
With this in mind, equation \eqref{cost_control+estimation} can be seen as a trade-off  coming from the conversion of a bi-objective problem into a mono-objective one. 
The last remarks lead to the following optimisation problem for $\tilde{i}_0\in\mathbb{R}^{n_y}$:

\begin{equation}
\begin{array}{rrclcc}
\widetilde{V}(\tilde{i}_0)=\displaystyle \underset{\pi_{0}^c,\pi_{0}^e}{\text{min}} & \multicolumn{3}{l}{E^{\pi}_{p_0}\left[\sum_{k=0}^{+\infty} \alpha^k g^{aug}(X_k,W_k,\xi_k)| \tilde{I}_0 = \tilde{i}_0\right]} \\
\textrm{s.t.} & X_{k+1} & = & f^{aug}(X_{k},W_{k},{\xi}_{k}), \\
 &Y_k  &=&  h(X_k,\eta_k),\\
&\tilde{I}_{k+1}&=&(\tilde{I}_k,W_{k},Y_{k+1}),\\
& W_k & = & (\pi_k^c(\tilde{I}_k),\pi_k^e(\tilde{I}_k)), \; \forall k\geq0.
\end{array}
\label{sto_control+estimation_imperfect_inf_infinite_horizon}
\end{equation}

Problem \eqref{sto_control+estimation_imperfect_inf_infinite_horizon} combines Problems \eqref{optimal_estimation_gen_inf_vec} and \eqref{sto_opt_control_imperfect_inf_infinite_horizon} in one multistage optimisation problem. The study of this formulation has been started in \cite{flayac_nonlinear_2017}. It is actually inspired from \cite{copp_nonlinear_2014} in which a similar gathering is done in a min-max optimisation framework. We would like to write the DP principle for the problem \eqref{sto_control+estimation_imperfect_inf_infinite_horizon} to make explicit links with classical optimal control and optimal estimation. As we did for the problem \eqref{sto_opt_control_imperfect_inf_infinite_horizon}, we can rewrite the problem \eqref{sto_control+estimation_imperfect_inf_infinite_horizon} in terms of the conditional distribution of $X_k$ knowing $\tilde{I}_k$, denoted by $\tilde{\mu}_k$. As for $\mu_k$, one can derive the dynamics of $\tilde{\mu}_k$ from equations \eqref{obs_sto_gen} and \eqref{dyn_sto_augmented}:
\begin{align}
    \tilde{\mu}_{k+1}&=F^{aug}\left(\tilde{\mu}_k,Y_{k+1},W_k\right),\label{filtering_equation_red_augmented}
\end{align}
with  $F^{aug}$ : $\mathcal{P}(\mathbb{R}^{n_x})\times\mathbb{R}^{n_y}\times\mathbb{R}^{n_{u}} \longrightarrow \mathcal{P}(\mathbb{R}^{n_x})$ leading to:
\begin{align}
\begin{array}{rrclcc}
\widetilde{V}(\mu)=\displaystyle \underset{\pi_{0}^{aug}}{\text{min}} & \multicolumn{3}{l}{E\left[\sum_{k=0}^{+\infty}\alpha^k \tilde{g}^{aug}(\tilde{\mu}_k,W_k)| \tilde{\mu}_0 = \mu\right]} \\
\textrm{s.t.} &\tilde{\mu}_{k+1}&=&F^{aug}\left(\tilde{\mu}_k,Y_{k+1},W_k\right),\\
& W_k & = & \pi_{0}^{aug}(\tilde{\mu}_k),\;\forall k\geq 0,\\
\end{array}
\label{sto_control+estimation_imperfect_inf_infinite_horizon_state_filter}
\end{align}
where $\tilde{g}^{aug}=\tilde{g}_c+\tilde{g}_e$ with $\tilde{g}_c$ and $\tilde{g}_c$ defined as in Sections \ref{sec_optimal_estimation} and \ref{sec_optimal_control}.

One would like to use the structure of the cost to separate the problem of control and estimation. In this sense, the formulation \eqref{sto_control+estimation_imperfect_inf_infinite_horizon_state_filter} is not practical because it involves $\tilde{\mu}_k$ which depends on the estimator whereas $f^{aug}$ actually does not. To split the two problems back, we start by noticing that the filtering equation of the augmented system and of the original one are the same meaning that:
\begin{align*}
   F^{aug}\left(\tilde{\mu}_k,Y_{k+1},W_k\right)=F\left(\mu_k,Y_{k+1},U_k\right).
\end{align*}
Note that $\tilde{\mu}_0=\mu_0$. This leads, by recursion on $k$, to $\tilde{\mu}_k=\mu_k$ almost surely. 
Finally, we can write our coupled control and estimation problem as a perfect information problem with $\mu_k$ as the state: 
\begin{align}
\begin{array}{rrclcc}
\widetilde{V}(\mu)=\displaystyle \underset{\pi_{0}^c,\pi_{0}^e }{\text{min}} & \multicolumn{3}{l}{E^{\pi}\left[\sum_{k=0}^{+\infty}\alpha^k( \tilde{g}^c(\mu_k,U_k)+\tilde{g}^e(\mu_k,\widehat{X}_k))| \mu_0 = \mu\right]} \\
\textrm{s.t.} &\mu_{k+1}&=&F\left(\mu_k,Y_{k+1},U_k\right),\\
& U_k & = & \pi_0^c(\mu_k), \\
& \widehat{X}_k & = & \pi_0^e(\mu_k), \; \forall k\geq 0.\\
\end{array}
\label{sto_control+estimation_imperfect_inf_infinite_horizon_state_filter_detail}
\end{align}

\subsection{Main result}\label{sec_dp_control+estimation}

Theorem \ref{thm:optimal_policies_coupled} shows that the augmented optimal policies of Problem \eqref{sto_control+estimation_imperfect_inf_infinite_horizon_state_filter_detail} can be decomposed into the solution of an optimal estimation problem and the solution of a optimal control problem on  $\mathcal{P}(\mathbb{R}^{n_x})$ whose cost depends on the optimal value of the estimation problem.  

\begin{theorem}\label{thm:optimal_policies_coupled}
For any $\mu \in \mathcal{P}(\mathbb{R}^{n_x})$, $\pi_{aug}^*=(\pi_e^*,\pi_c^*)$ is a solution of Problem \eqref{sto_control+estimation_imperfect_inf_infinite_horizon_state_filter_detail} if and only if $\pi_e^*$ is a solution of the following optimal estimation problem:

\begin{align*}
    \begin{array}{rrclcc}
\displaystyle \underset{\pi^e_k}{\emph{min}} & \multicolumn{3}{l}{\tilde{g}^e(\mu_k,\widehat{X}_k)} \\
\textrm{s.t.} & \widehat{X}_k  =  \pi^e_0(\mu_k),
\end{array}
\end{align*}

and $\pi_c^*$ is a solution of the following stochastic optimal problem on the space of probability measures:

  \begin{align}
\begin{array}{rrclcc}
\widetilde{V}(\mu)=\displaystyle \underset{\pi_{0}^c }{\text{min}} & \multicolumn{3}{l}{E^{\pi}\left[\sum_{k=0}^{+\infty}\alpha^k( \tilde{g}^c(\mu_k,U_k)+\tilde{g}^e_*(\mu_k))| \mu_0 = \mu\right]} \\
\textrm{s.t.} &\mu_{k+1}&=&F\left(\mu_k,Y_{k+1},U_k\right),\\
& U_k & = & \pi_0^c(\mu_k), \; \forall k\geq 0.\\
\end{array}
\label{sto_control_imperfect_inf_infinite_horizon_state_filter_estimation_cost}
\end{align}
Aditionally, the value function of PORblem
\end{theorem}

From \eqref{sto_control+estimation_imperfect_inf_infinite_horizon_state_filter_detail}, the Bellman equation of the coupled problem reads:
\begin{align}
&\begin{array}{rrclcc}
\widetilde{V}(\mu)=\displaystyle \underset{(u,\hat{x})\in \mathcal{U}\times\mathbb{R}^{n_x}}{\text{min}} & \multicolumn{3}{l}{E\left[ \tilde{g}^c(\mu,u) +\tilde{g}^e(\mu,\hat{x}) + \alpha \widetilde{V}(\mu_{\ell+1})| \mu_\ell = \mu\right]} \\
\textrm{s.t.} &\mu_{\ell+1}=F\left(\mu_\ell,Y_{\ell+1},u\right),
\end{array}
\notag\\
\intertext{As $\tilde{g}^c(\mu,u)$ and $\tilde{g}^e(\mu,\hat{x})$ are deterministic and $F$ does not depend on $\hat{x}$:}
&\begin{array}{rrclcc}
\widetilde{V}(\mu)=\displaystyle \underset{u \in \mathcal{U}}{\text{min}} & \multicolumn{3}{l}{  \left(\underset{\hat{x}\in \mathbb{R}^{n_x}}{\text{min}}\tilde{g}^e(\mu,\hat{x})\right) +\tilde{g}^c(\mu,u) + \alpha E\left[\widetilde{V}(\mu_{\ell+1})| \mu_\ell = \mu\right]} \\
\textrm{s.t.} &\mu_{\ell+1}=F\left(\mu_\ell,Y_{\ell+1},u\right).
\end{array}
\label{sto_dp_control+estimation_infinite_horizon_state_filter}
\end{align}
Equation  \eqref{sto_dp_control+estimation_infinite_horizon_state_filter} illustrates the fact that Problem \eqref{sto_control+estimation_imperfect_inf_infinite_horizon}, which gathers optimal control and optimal estimation, can actually be split back into a hierarchy of two problems. Indeed, it justifies the use of a resolution scheme in two steps:
\begin{enumerate}
    \item First, one solves the inner minimisation in Problem \eqref{sto_dp_control+estimation_infinite_horizon_state_filter} which is a classical optimal problem of the form \eqref{optimal_estimation_gen_state_filter}. Any of its solution gives a time invariant optimal estimation policy denoted by $\pi^e_*(\mu)$. From this, we set:
    \begin{align}
        \tilde{g}^e_*(\mu):=\tilde{g}^e(\mu,\pi^e_*(\mu))= \underset{\hat{x}\in \mathbb{R}^{n_x}}{\text{min}}\tilde{g}^e(\mu,\hat{x}) \label{inner_optimal_cost_star}
    \end{align}
   \label{optimal_estimation_step}
    \item Secondly, by substituting \eqref{inner_optimal_cost_star} in \eqref{sto_dp_control+estimation_infinite_horizon_state_filter} one gets:
    \begin{align}
    &\begin{array}{rrclcc}
        \widetilde{V}(\mu)=\displaystyle \underset{u \in \mathcal{U}}{\text{min}} & \multicolumn{3}{l}{  \tilde{g}^e_*(\mu) +\tilde{g}^c(\mu,u) + \alpha E\left[\widetilde{V}(\mu_{\ell+1})| \mu_\ell = \mu\right]} \\
        \textrm{s.t.} &\mu_{\ell+1}=F\left(\mu_\ell,Y_{\ell+1},u\right).
    \end{array}
    \label{sto_dp_control_infinite_horizon_state_filter_estimation_cost}
    \end{align}
    Equation \eqref{sto_dp_control_infinite_horizon_state_filter_estimation_cost} can be interpreted as the Bellman equation of a stochastic optimal control problem on $\mathcal{P}(\mathbb{R}^{n_x})$ which has the same optimal value as Problem \eqref{sto_control+estimation_imperfect_inf_infinite_horizon_state_filter}. The second step is then to solve this problem:
    \begin{align}
\begin{array}{rrclcc}
\widetilde{V}(\mu)=\displaystyle \underset{\pi_{0}^c }{\text{min}} & \multicolumn{3}{l}{E^{\pi}\left[\sum_{k=0}^{+\infty}\alpha^k( \tilde{g}^c(\mu_k,U_k)+\tilde{g}^e_*(\mu_k))| \mu_0 = \mu\right]} \\
\textrm{s.t.} &\mu_{k+1}&=&F\left(\mu_k,Y_{k+1},U_k\right),\\
& U_k & = & \pi_0^c(\mu_k), \; \forall k\geq 0.\\
\end{array}
\label{sto_control_imperfect_inf_infinite_horizon_state_filter_estimation_cost}
\end{align}
   Note that $\tilde{g}^e_*(\mu)=\langle\mu,g^e(\cdot,\pi^e_*(\mu)) \rangle$ so $\tilde{g}^e_*$ is generally nonlinear in $\mu$ and is not an integral w.r.t. $\mu$ as in \eqref{sto_opt_control_imperfect_inf_infinite_horizon_state_filter}. Therefore, Problem \eqref{sto_control_imperfect_inf_infinite_horizon_state_filter_estimation_cost} cannot be written back in the form of Problem \eqref{sto_opt_control_imperfect_inf_infinite_horizon}. \label{optimal_control_estimation_cost_step}
\end{enumerate}
At first sight, this split scheme looks like any classical one because most of the outputfeedback controllers are built from an estimation and a control step. However, it is different from a classical scheme in several ways. First, the two hierarchical steps appears naturally from the coupled problem \eqref{sto_control+estimation_imperfect_inf_infinite_horizon} meaning that the splitting is structural in this case and does not come from an assumption of separation. Secondly, the actual value of $\widehat{X}_k$ is not directly involved in the control problem \eqref{sto_control_imperfect_inf_infinite_horizon_state_filter_estimation_cost} but only $\mu_k$. In practice however, the same approximation of $\mu_k$ is generally used both in the estimation and the control step. Finally, the Problem \eqref{sto_control_imperfect_inf_infinite_horizon_state_filter_estimation_cost} has new interesting properties that will be the topic of Section \ref{sec_modified_explicit}.
The goal of the next two sections is to give a practical resolution of the scheme described in this section by means of a particle filtering algorithm for Step \ref{optimal_estimation_step} and of a general  Explicit dual control scheme for Step \ref{optimal_control_estimation_cost_step}.

\section{Particle filtering and  near-optimal estimation}\label{sec_particle_near_opti}
In equation \eqref{control_policy_info_vec_imperf_inf_state_filter_time_homo}, the posterior distribution $\mu_k$ is implicitly assumed known when $I_k$ is. However, for a general nonlinear case, $\mu_k$ does not have an analytical form and approximations must be carried out. Kalman filters are widespread and easy-to-compute approximations of $\mu_k$ but they may fail in the presence of high nonlinearities and multimodality. Besides, nonlinear Kalman filters are not optimal even for problem \eqref{optimal_estimation_MSE_cond} and their suboptimality is generally impossible to quantify. That is why, in the sequel, we focus on particle filtering algorithms to approach $\mu_k$. Then, we show that the empirical mean of a specific particle filter is near-optimal for Problem \eqref{optimal_estimation_MSE_cond} and a rate of convergence is provided.
\subsection{Particle filtering}

A particle filter approximates the posterior distribution  $\mu_k$  by a set of $N\geq1$ particles, ${\left(x^{i}_{k}\right)}_{i= 1,..,N }$ valued in $\mathbb{R}^{n_x}$, associated with nonnegative and normalized weights ${\left({\omega}^{i}_{k}\right)}_{i=1,..,N }$. This approximation is denoted by $\mu_k^N$. The same can be done with the predicted distribution $\mu_{k|k-1}$ with similar notations:
 \begin{align}
      \mu_k^N=\sum_{i=1}^{N}\omega_k^{i}\delta_{x^{i}_k},\qquad
       \mu_{k|k-1}^N=\sum_{i=1}^{N}\omega_{k|k-1}^{i}\delta_{x^{i}_{k|k-1}}\label{particle_filter}.
\end{align}
 As for Kalman filters, a particle filter is computed recursively following two steps: prediction and correction. During the prediction step, the particles are propagated using an importance distribution that is often chosen as the Markov kernel $K$ from the dynamics. During the correction step, the weights are updated thanks to the last observation and the particles are resampled from the updated weights. In the sequel, we consider a particular algorithm coming from \cite{hu_general_2011} where an intermediary step of selection of the particles according to their likelihood is added. Moreover, we define the empirical mean of the filter $\mu_k^N$, denoted by $\widehat{X}_k^N$ as follows:
\begin{align}
    \widehat{X}_k^N &=\sum_{i=1}^{N}\omega_k^{i}x_k^{i}=\langle\mu_k^N,\mathrm{Id} \rangle.
    \label{empirical_mean_integral}
 \end{align}
\subsection{Near-optimal estimation}
\subsubsection{Statement of the problem}
In this section, we consider a fixed vector of information $i_k$. Thus, $\mu_k$ is the distribution of $X_k$ conditionally to $I_k=i_k$ and $\mu_{k|k-1}$ is the distribution of $X_k$ conditionally to $I_{k-1}=i_{k-1}$.
In Section \ref{sec_coupled_estimation_control}, we have seen that if we model the problem of control and estimation as an optimisation problem then a step of optimal estimation is required at each time $k$. We focus on the Mean Square Error minimisation problem i.e. we assume that  $g^e(x,\hat{x})={\Vert \hat{x}-x\Vert}^2$.  We recall the problem seen in Section \ref{sec_optimal_estimation}:
\begin{align}
&\underset{\hat{x}\in \mathbb{R}^{n_x}}{\text{min}}E[{\Vert \hat{x}-X_k\Vert}^2|I_k=i_k]. \label{inner_optimal_estimation_MSE}
    \end{align}

 We are concerned with MSE because it is a very popular estimation error measure of particle filters in practice but it has not really been studied theoretically. Besides, in this section, we study  the conditional MSE minimisation problem \eqref{optimal_estimation_MSE_cond} precisely because this problem appears in Step \ref{optimal_estimation_step} in Section \ref{sec_coupled_estimation_control}. However, we are also interested in the minimisation of the total MSE. In fact, It gives a more useful assessment of the filter than the conditional MSE because it does not depend on $i_k$ which is unknown at the initial time.

From this, we define the conditional and total optimal MSE at time $k$ denoted respectively by $e^{cond}_{k,*}$ and $e^{tot}_{k,*}$, as follows, for any $i_k$:
\begin{align*}
   e^{cond}_{k,*}(i_k):= E\left[{\Vert X_k - \widehat{X}_k^*\Vert}^2\vert I_k=i_k\right],\qquad
   e^{tot}_{k,*}:= E\left[{\Vert X_k - \widehat{X}_k^*\Vert}^2\right].
\end{align*}

Similarly to the optimal MSE, we define the empirical conditional and total MSE associated with $\widehat{X}_k^N$ denoted respectively by $e^{cond}_{k,N}$ and $e^{tot}_{k,N}$, as follows, for any $i_k$:
\begin{align*}
   e^{cond}_{k,N}(i_k):= E\left[{\Vert X_k - \widehat{X}_k^N\Vert}^2\vert I_k=i_k\right],\qquad
   e^{tot}_{k,N}:= E\left[{\Vert X_k - \widehat{X}_k^N\Vert}^2\right],
\end{align*}
where the expectation is also taken over the randomness of the particles.

The main contribution of the sequel, gathered in Theorems \ref{prop_convergence_cond_MSE} and \ref{prop_bound_total_mse} is to show that under suitable assumptions on the dynamics \eqref{dyn_sto_gen}, the observation equation \eqref{obs_sto_gen} and the particle filter \eqref{particle_filter}, the empirical MSE converges to the optimal MSE as the number of particles goes to infinity. More precisely, we prove  error bounds between  $e^{cond}_{k,N}(i_k)$ and $ e^{cond}_{k,*}(i_k)$ and between  $e^{tot}_{k,N}$ and $e^{tot}_{k,*}$.

\subsubsection{Error bounds between the optimal MSE and the empirical MSE} \label{sec_error_bound_cond_MSE}
The main difficulty in the following comes from the fact that, even if $\widehat{X}_k^N$ is very commonly used as an approximation of $\widehat{X}_k^*$, estimating rigorously the  convergence of $\widehat{X}_k^N$ to $\widehat{X}_k^*$ cannot be achieved by classical error bounds on particle filters. In fact, from equations \eqref{conditional_expectation_integral} and \eqref{empirical_mean_integral}, $\widehat{X}_k^*$ and $\widehat{X}_k^N$ are the integral of Id, which is unbounded,  w.r.t. $\mu_k$ and $\mu_k^N$. Therefore, it does not fit in the classical framework of weak error bounds. 

To begin with, without assumptions, we can compare the several MSE  using the optimality of $e^{cond}_{k,*}$ and $e^{tot}_{k,*}$. This is the topic of Lemma \ref{prop_optimality_MSE} which is a direct consequence of the optimality  of $\widehat{X}_k^*$.

\begin{lemma}\label{prop_optimality_MSE}
The following inequalities hold, for any $k\geq0$ and any $i_k$:
\begin{align*}
    e^{cond}_{k,*}(i_k)\leq e^{cond}_{k,N}(i_k), \qquad
    e^{tot}_{k,*}\leq e^{tot}_{k,N}.\\
\end{align*}
\end{lemma}

Thus, considerations of optimality give a lower bound on both the conditional and total MSE of the particle filter. To find upper bounds, we first treat the conditional case and extend it later to the total case. 
\paragraph{Bounds on the conditional MSE}

This section is dedicated to the proof of an upper bound on $e^{cond}_{k,N}(i_k)$ and of the convergence of $e^{cond}_{k,N}(i_k)$ to  $e^{cond}_{k,*}(i_k)$. The main result of this section is contained in Theorem \ref{prop_convergence_cond_MSE}.  First, before stating the results of this section, we would like to stress the particularity of our problem by rewriting $e^{cond}_{k,N}(i_k)$ as follows:
\begin{align}
    e^{cond}_{k,N}(i_k)=E\left[{\Vert X_k - \widehat{X}_k^N\Vert}^2\vert I_k=i_k\right]= E\left[{\Vert X_k -\widehat{X}_k^*+\widehat{X}_k^*- \widehat{X}_k^N\Vert}^2\vert I_k=i_k\right],
\end{align}
   By  Young's inequality, we get for any $\epsilon > 0$:
\begin{align}
    E\left[{\Vert X_k - \widehat{X}_k^N\Vert}^2\vert I_k=i_k\right]&\leq\left(1+\epsilon\right)E\left[{\Vert X_k - \widehat{X}_k^*\Vert}^2\vert I_k=i_k\right]\\
    &+\left(1+\frac{1}{\epsilon}\right)E\left[{\Vert \widehat{X}_k^* - \widehat{X}_k^N\Vert}^2\vert I_k=i_k\right],\notag\\
     &\leq\left(1+\epsilon\right)E\left[{\Vert X_k - \widehat{X}_k^*\Vert}^2\vert I_k=i_k\right]\notag\\
     &+\left(1+\frac{1}{\epsilon}\right)E\left[{\Vert \langle\mu_k,\mathrm{Id} \rangle - \langle\mu_k^N,\mathrm{Id} \rangle\Vert}^2\vert I_k=i_k\right].\notag
     \end{align}
The last inequality can be rewritten as follows:
\begin{align}
          e^{cond}_{k,N}(i_k)&\leq(1+\epsilon)e^{cond}_{k,*}(i_k)+\left(1+\frac{1}{\epsilon}\right)e^{filter}_{k,N}(i_k),\label{upper_bound_empirical_MSE}
\end{align}
where $e^{filter}_{k,N}(i_k)=E\left[{\Vert \langle\mu_k,\mathrm{Id} \rangle - \langle\mu_k^N,\mathrm{Id} \rangle\Vert}^2\vert I_k=i_k\right]$.
One can deduce from equation \eqref{upper_bound_empirical_MSE} that, up to $\epsilon$, it is sufficient to control the term $e^{filter}_{k,N}(i_k)$. For $x=(x_1,\dots,x_{n_x})$ in the canonical basis of $\mathbb{R}^{n_x}$, and $j=1,..,{n_x}$ one defines  the $j^{th}$ coordinate function, $\phi_j$, such that, $\forall x \in \mathbb{R}^{n_x} $:
\begin{align*}
    \phi_j(x)=x_j.
\end{align*}
Then, $e^{filter}_{k,N}(i_k)$ can be decomposed as follows:
\begin{align}
    e^{filter}_{k,N}(i_k)&=E\left[\sum_{j=1}^{n_x}{\vert \langle\mu_k,\phi_j \rangle - \langle\mu_k^N,\phi_j \rangle\vert}^2\vert I_k=i_k\right],\notag\\
    &=\sum_{j=1}^{n_x}E\left[{\vert \langle\mu_k,\phi_j \rangle - \langle\mu_k^N,\phi_j \rangle\vert}^2\vert I_k=i_k\right].\label{upper_bound_empirical_MSE_coordinates}
\end{align}

This term can be seen as a quadratic error term of the particle filter $\mu_k^N$ where the scalar test functions are the coordinate maps. Actually, classical error bounds do not deal with unbounded functions like Id, see \cite{crisan_survey_2002} for a survey. Still, a class of unbounded functions has been treated in the form of a Central Limit Theorem in \cite{del_moral_monte-carlo_2002} but the result of convergence in law is too weak to be applied to $e^{filter}_{k,N}(i_k)$.  However, in \cite{hu_general_2011}, a bound on the $L^p$-norm for a class of potentially unbounded test functions is given. The error bound is written conditionally to $i_k$ in the following form:
\begin{align}
    E\left[{\vert \langle\mu_k,\psi \rangle - \langle\mu_k^N,\psi \rangle\vert}^p\vert I_k=i_k\right]\leq C_k\frac{\Vert \psi\Vert_{k,p}^p}{N^{p-p/r}},\label{error_bound_unbounded_article}
\end{align}
where $p\geq2$, $1\leq r\leq2$, $C_k$ is a coefficient depending on $i_k$, $\psi$ is a test function and  ${\Vert\psi\Vert}_{k,p}=(\text{max}(1,{\langle\mu_0,{\vert\psi\vert}^p \rangle}^{\frac{1}{p}},\dots,{\langle\mu_k,{\vert\psi\vert}^p \rangle}^{\frac{1}{p}})$. It was originally written conditionally to a sequence of observation $y_{0:k}$ but the extension conditionally to $i_k$ is straightforward.
In the sequel, we would like to apply the bound \eqref{error_bound_unbounded_article} to $e^{filter}_{k,N}(i_k)$. To do so, we present the adapted assumptions of \cite{hu_general_2011}  for the particular case $\psi=\phi_j$, $p=2$ and $r=2$.
\begin{assumption} \label{as:gamma_1}
For any $k\geq1$, for $0<\epsilon_k<1$,  for a.a. $i_k$,  there exists $ N_k(i_k)>0$ such that, for $N\geq N_k(i_k)$:
\begin{align}
\gamma_k=\underset{i_k}{\emph{inf}}\langle\mu_{k\vert k-1},\rho \rangle&>0, \label{high_likelihood_filter_opt}\\
    {P}( \langle\mu_{k\vert k-1}^N,\rho \rangle>\gamma_k\vert I_k=i_k)&\geq 1-\epsilon_k.\label{high_likelihood_particle_filter}
\end{align}
\end{assumption}
\noindent In particular, under Assumption \ref{as:gamma_1}, for a.a. $i_k$, we have $ \langle\mu_{k\vert k-1},\rho \rangle>\gamma_k$ with $\gamma_k$ independent of $i_k$.
\begin{assumption} \label{as:finite}
For $k\geq1$ and for a.a. $y_{k}$, $x_{k}$, $x_{k-1}$ and $u_{k-1}$:
\begin{align*}
   \rho(y_k,x_k)&<+\infty,  \qquad K(x_k,x_{k-1},u_{k-1})<+\infty.
\end{align*}
\end{assumption}
For $j =1,..,n_x$, we denote respectively the $L^{\infty}$-norm of $K$, $\rho$ and $\rho\phi_j$ w.r.t. $x$ by $ \Vert K\Vert$, $ \Vert \rho\Vert$ and $ \Vert \rho\phi_j\Vert$ i.e, for a.a. $ i_k$:
     \begin{align*}
      \Vert K\Vert&=\underset{x_0,x_1,}{\textrm{sup}}\; K(x_1,x_0,u_{k-1})\\
          \Vert \rho\Vert&=\underset{x}{\textrm{sup}}\; \rho(y_k,x),\\
           \Vert \rho\phi_j^2\Vert&=\underset{x}{\textrm{sup}}\; \vert\phi_j^2(x)\rho(y_k,x)\vert.
     \end{align*}
\begin{assumption} \label{as:bounded}
For $k\geq1$ and for a.a. $i_k$,
\begin{align*}
    \Vert K\Vert<+\infty,\qquad \Vert \rho\Vert <+\infty,\qquad \Vert \rho\phi_j^2\Vert <+\infty.
\end{align*}
\end{assumption}

As $ \Vert \rho\phi_j\Vert \leq {\Vert \rho\phi_j^2\Vert}^{\frac{1}{2}} {\Vert \rho\Vert}^{\frac{1}{2}} $, Assumption \ref{as:bounded}  implies that $ \Vert \rho\phi_j\Vert<+\infty$.
Lemma \ref{prop_upper_bound_cond_mse} presents then an upper bound of $e^{cond}_{k,N}(i_k)$.

\begin{lemma}\label{prop_upper_bound_cond_mse}

Under Assumption \ref{as:gamma_1}, \ref{as:finite} and \ref{as:bounded}, for $\epsilon>0$, for $ j =1,..,n_x$ and $k\geq0$, for a.a. $i_k$, there exist $ C_{k,j}>0$ and $ M_{k,j}>0$, such that for $ N\geq N_k(i_k)$: 
\begin{align}
e^{cond}_{k,N}(i_k)&\leq(1+\epsilon)e^{cond}_{k,*}(i_k)+\left(1+\frac{1}{\epsilon}\right)\frac{\sum_{j=1}^{n}C_{k,j}{\Vert\phi_j\Vert}_{k,2}^2}{N}, \label{upper_bound_empirical_MSE_convergence}
\end{align}
where ${\Vert\phi_j\Vert}_{k,2}=\emph{max}(1,{\langle\mu_0,{\vert\phi_j\vert}^2 \rangle}^{\frac{1}{2}},\dots,{\langle\mu_k,{\vert\phi_j\vert}^2 \rangle}^{\frac{1}{2}})$.
Besides, $C_{k,j}$ and $M_{k,j}$ follow the following coupled recursion, for a.a. $i_k$:
 
 \begin{align}
     M_{0,j}&=3,\label{M_ini}\\
     C_{0,j}&=8\widetilde{C},\label{C_ini}\\
     M_{k,j}&=2+\alpha_{k,j}\left(1+\left(\frac{4-\epsilon_k}{1-\epsilon_k}+1\right)M_{k-1,j}\right)\label{M_current}\\
     C_{k,j}^{\frac{1}{2}}&=2^{\frac{3}{2}}{(\widetilde{C})}^{\frac{1}{2}}M_{k,j}^{\frac{1}{2}}+\frac{2^{\frac{3}{2}}{(\widetilde{C})}^{\frac{1}{2}}\beta_{k,j}}{{(1-\epsilon_k)}^{\frac{1}{2}}}M_{k-1,j}^{\frac{1}{2}}\label{C_current}\\
     &+\frac{{\Vert K\Vert}^{\frac{3}{2}}{\Vert\rho\Vert}_{k,2}\beta_{k,j}}{(1-\epsilon_k)\vert\frac{\gamma_k}{2}-\langle\mu_{k\vert k-1},\rho \rangle\vert}M_{k-1,j}^{\frac{1}{2}}C_{k-1,j}^{\frac{1}{2}}
     +\Vert K \Vert\beta_{k,j}C_{k-1,j}^{\frac{1}{2}},\notag
     \end{align}
     
     \begin{align}
          \beta_{k,j}=\frac{ \Vert \rho\Vert( \Vert \phi_j\rho\Vert+\frac{\gamma_k}{2})}{\frac{\gamma_k}{2}\langle\mu_{k\vert k-1},\rho \rangle},\qquad
     \alpha_{k,j}={\Vert K\Vert}^2\frac{ \Vert \rho\Vert( \Vert \phi_j^2\rho\Vert+\frac{\gamma_k}{2})}{\frac{\gamma_k}{2}\langle\mu_{k\vert k-1},\rho \rangle}\label{beta_current},
     \end{align}

  where:
$\widetilde{C}>0$, ${\Vert\rho\Vert}_{k,2}=\langle\mu_k,{\rho} \rangle \leq  \Vert\rho\Vert$
and $N_k(i_k)\geq\frac{{\Vert\rho\Vert}_{k,2}^2{\Vert K\Vert}^2\underset{j}{\max}\; C_{k-1,j} }{{\vert\frac{\gamma_k}{2}-\langle\mu_{k\vert k-1},\rho \rangle\vert}^2 \epsilon_k}$ with $0<\epsilon_k<1$ is fixed independently of $i_k$.


\end{lemma}
\begin{proof}
See Appendix \ref{appendix_1}.
\end{proof}
  
By combining Lemma \ref{prop_optimality_MSE} and \ref{prop_upper_bound_cond_mse}, one finally gets Theorem \ref{prop_convergence_cond_MSE}.

\begin{theorem}
\label{prop_convergence_cond_MSE}
Under Assumption \ref{as:gamma_1}, \ref{as:finite} and \ref{as:bounded}, for $\epsilon>0$, for $k\geq0$, for a.a. $i_k$, there exists $ C_{k,j}>0$ such that $\forall N\geq N_k(i_k)$: 
\begin{align}
e^{cond}_{k,*}(i_k)\leq e^{cond}_{k,N}(i_k)&\leq(1+\epsilon)e^{cond}_{k,*}(i_k)+\left(1+\frac{1}{\epsilon}\right)\frac{\sum_{j=1}^{n}C_{k,j}{\Vert\phi_j\Vert}_{k,2}^2}{N}. \label{bound_cond_mse}
\end{align}
In particular,  for a.a. $i_k$:
\begin{align}
   E\left[{\Vert X_k - \widehat{X}_k^N\Vert}^2\vert I_k=i_k\right]&\underset{N\rightarrow+\infty}{\longrightarrow}E\left[{\Vert X_k - \widehat{X}_k^*\Vert}^2\vert I_k=i_k\right].\label{limit_bound_error_2}
\end{align}
\end{theorem}
\begin{proof}
See Appendix \ref{appendix_2}.
\end{proof}
\begin{remark}\hfill

 Assumption \ref{as:finite} is very mild because most systems have finite likelihood and transition kernel. Assumption \ref{as:bounded}  requires that the likelihood function $\rho$ vanishes sufficiently when $\Vert x\Vert \rightarrow +\infty$ for a fixed vector $i_k$ to counter the increasing effect of $\phi_j$ as explained in \cite{hu_basic_2008}. It is typically verified with Gaussian measurement noise. Assumption \ref{as:gamma_1} is natural in particle filtering. It requires that the predicted {distribution} $\mu_{k|k-1}$ and the predicted \emph{particles} match the likelihood $\rho$ for each information vector $i_k$. The failure of this property a well known issue in particle filtering and is studied in more depth in \cite{hu_basic_2008}, \cite{le_gland_stability_2004} and \cite{crisan_survey_2002}. It is notably showed that it has an impact on the precision of some error bounds.
    Intuitively,  the algorithm from  \cite{hu_general_2011} forces the particles to be positioned where the true state is sufficiently likely to be with respect to the new observation $y_k$.
    
 Theorem \ref{prop_convergence_cond_MSE} basically means that $\widehat{X}_k^N$ is  near-optimal with respect to the conditional MSE when the number of particle is large enough. Besides, equation \eqref{bound_cond_mse} gives an estimation of the speed of convergence of the empirical MSE. For example, for $q=\frac{1}{2}$, one can see that the speed of convergence of this MSE is of order $\frac{1}{\sqrt{N}}$. It is slower than usual in Monte Carlo methods. One would rather expect a convergence rate of order $\frac{1}{N}$. The conservativeness of the bound \eqref{bound_cond_mse} comes from our use of Young's inequality instead of Minkowsky's inequality. Actually, a very similar reasoning could be undertaken using Minkowsky's inequality and one would get a better convergence rate but it would involve the conditional Root Mean Square Errors (RMSE), ${(e^{cond}_{k,*})}^{\frac{1}{2}}$ and ${(e^{cond}_{k,N})}^{\frac{1}{2}}$, and not the MSE. The RMSE is easier to interpret than the MSE in practice in an estimation context for the same reasons that the standard deviation is easier to relate to concrete data than the variance. However, minimising a MSE is more adapted to the context of stochastic optimisation defined in Section \ref{sec_dp_control+estimation}. That is why, we focus on MSEs and not RMSEs in this section even if we lose some precision in the error bounds. 

\end{remark}

 In the sequel, we would like to extend the result of Theorem \ref{prop_convergence_cond_MSE} to the total MSE. A intuitive way would be to integrate equation \eqref{bound_error_1} over $i_k$. However, it is not possible in its current form because $C_{k,j}$ and $M_{k,j}$ depend on $i_k$ which makes the integrability on the right-hand of equation \eqref{bound_error_1} hard to evaluate. Moreover, the threshold $N_k(i_k)$ also depends on $i_k$ so one cannot apply the Dominated Convergence theorem directly.
 
 \paragraph{Bound on the total MSE}

 The main contribution of this section is the extension of the result of Theorem \ref{prop_convergence_cond_MSE} to the total MSE. This result is presented in Theorem \ref{prop_bound_total_mse}. 
 
 We first assume that $X_k$ is square integrable
\begin{assumption}\label{as:square_int}
For any $ k\geq0$:
\begin{align*}
  \mathbb{E}[{\Vert X_k \Vert}^2]<+\infty.  
\end{align*}
\end{assumption}

As in the conditional case, Lemma \ref{prop_optimality_MSE} provides a lower bound of $e^{tot}_{k,N}$, we are then looking for an upper bound of  $e^{tot}_{k,N}$. As suggested earlier, one would like to integrate the right-hand  side of equation \eqref{upper_bound_empirical_MSE_convergence} w.r.t. $i_k$, which is defined, for $ \epsilon>0$, for a.a. $i_k$ and for any $ N\geq N_k(i_k)$ by:
\begin{align*}
     (1+\epsilon)e^{cond}_{k,*}(i_k)+\left(1+\frac{1}{\epsilon}\right)\frac{\sum_{j=1}^{n}C_{k,j}{\Vert\phi_j\Vert}_{k,2}^2}{N}.
\end{align*}

Note first that, under Assumption \ref{as:square_int}, $E[e^{cond}_{k,*}(i_k)]=e^{tot}_{k,*}<+\infty$. 
and that $\forall k\geq0, \forall j = 1,..,n$, $\langle\mu_k,{\vert\phi_j\vert}^2 \rangle$
is integrable. Thus, ${\Vert\phi_j\Vert}_{k,2}^2$ is integrable. 

Actually, the first issue lays in the fact that $C_{k,j}$ depends on $i_k$ and it is not clear at all from equations  \eqref{M_ini} to \eqref{beta_current} that each term $C_{k,j}{\Vert\phi_j\Vert}_{k,2}^2$ is integrable w.r.t. $i_k$. To tackle this issue, we show that if the coefficients ${\Vert K\Vert}$, ${\Vert \rho\Vert}$, ${\Vert \rho\phi_j^2\Vert}$ and ${\Vert \rho\phi_j\Vert}$ are bounded uniformly w.r.t. $i_k$, then $C_{k,j}$ and $M_{k,j}$  are too. This leads to Assumption \ref{as:bounded_2}.


 \begin{assumption}  $\forall k \geq0$, $\forall j =1,..,n_x $:
 \label{as:bounded_2}
 \begin{align*}
  {\Vert K\Vert}_{\infty}&=\underset{x_1,x_0,u_0}{\emph{sup}}\; K(x_0,x_1,u_0)<+\infty,\\
     {\Vert \rho\Vert}_{\infty}&=\underset{x,y}{\emph{sup}}\; \rho(y,x)<+\infty,\\
           {\Vert \rho\phi_j^2\Vert}_{\infty}&=\underset{x,y}{\emph{sup}}\;\vert\phi_j^2(x)\rho(y,x)\vert<+\infty.
 \end{align*}
 \end{assumption}
 
 It is clear that Assumption \ref{as:bounded_2} implies Assumption \ref{as:bounded}.
 As $ {\Vert \rho\phi_j\Vert}_{\infty} \leq {\Vert \rho\phi_j^2\Vert}^{\frac{1}{2}}_{\infty} {\Vert \rho\Vert}^{\frac{1}{2}}_{\infty} $, Assumption \ref{as:bounded_2} implies that $ {\Vert \rho\phi_j\Vert}_{\infty}<+\infty$.

  
  We can now state the following Lemma:
  
  \begin{lemma}\label{prop_bound}
  Under Assumptions \ref{as:gamma_1}, \ref{as:finite}, \ref{as:square_int} and \ref{as:bounded_2}, for $k \geq0$, for $ j =1,..,n_x $, there exist $ C_{k,j}' >0$ and $M_{k,j}'>0$ such that, for a.a. $i_k$:
  \begin{align*}
       C_{k,j}\leq C_{k,j}'<+\infty,\qquad
        M_{k,j}\leq M_{k,j}'<+\infty.
  \end{align*}
  \end{lemma}
  
  \begin{proof}
  See Appendix \ref{appendix_3}.
  \end{proof}

   Finally, each term $C_{k,j}'{\Vert\phi_j\Vert}_{k,2}^2$ is integrable which solves our first problem. Our second problem is that the threshold  $N_k(i_k)$ also depends on $i_k$. Actually, under the same assumption, one can find a larger threshold that does not depend on $i_k$. This is the topic of the next result:
  \begin{theorem}\label{prop_bound_total_mse}
  Under Assumptions \ref{as:gamma_1}, \ref{as:finite} and \ref{as:bounded_2}, for $0<q<1$, for $ k\geq0$, there exists $ \bar{N}_k>0$, such that for any $ N\geq \bar{N}_k$
  \begin{align}
   e^{tot}_{k,*}\leq e^{tot}_{k,N}\leq\left(1+\frac{1}{N^q}\right)e^{tot}_{k,*}+&\left(1+N^q\right)\frac{\sum_{j=1}^{n}C_{k,j}'E\left[{\Vert\phi_j\Vert}_{k,2}^2\right]}{N}<+\infty\label{bound_error_2}.
\intertext{In particular:}
     E\left[{\Vert X_k - \widehat{X}_k^N\Vert}^2\right]&\underset{N\rightarrow+\infty}{\longrightarrow}E\left[{\Vert X_k - \widehat{X}_k^*\Vert}^2\right].
 \end{align}
  
  \end{theorem}
   
  \begin{proof}
  See Appendix \ref{appendix_4}.
  \end{proof}
  Similarly to Theorem \ref{prop_convergence_cond_MSE}, Theorem \ref{prop_bound_total_mse} means that the total MSE associated with $\widehat{X}^N_k$ is close to be optimal if the number of particle is sufficiently high. Theorem \ref{prop_bound_total_mse} also provides an estimation of the rate of convergence. This leads to several remarks and interpretations.
\begin{remark}\hfill

 Assumption \ref{as:square_int} is mild and needed to ensure that the optimal estimation error is finite. Assumption \ref{as:bounded_2} is stronger than Assumption \ref{as:bounded}. It is typically verified when the measurement noise is bounded. This will be illustrated in Section \ref{sec_example}.
 
 The coefficients $M_{k,j}'$ and $C_{k,j}'$ tend to $+\infty$ with time. Actually, It can be easily seen from equation \eqref{M_current'} that, for $ j=1,..,n_x$, and $k\geq1$, 
    \begin{align*}
        M_{k,j}'\geq \alpha_{k,j}'\theta_k M_{k-1,j}',
    \end{align*}
    with $\theta_k=1+\frac{4-\epsilon_k}{1-\epsilon_k}\geq2$.
    In fact, in most relevant cases, $\alpha_{k,j}'\geq1$ so $C_{k,j}'$ and $M_{k,j}'$ go to $+\infty$ as $k\rightarrow+\infty$.
     Thus, $\bar{N}_k$ tends to $+\infty$ too which means that the error bound from Theorem \ref{prop_bound_total_mse} is not uniform in $k$. This is classical for this type of error bound as described in \cite{crisan_survey_2002}. To get uniformity in $k$, one typically need a mixing assumption on $K$, see \cite{le_gland_stability_2004}.
Theorem \ref{prop_convergence_cond_MSE} justifies the use of particle in the framework of Section \ref{sec_dp_control+estimation} because it shows that $\widehat{X}_k^N$ solves approximately Problem \eqref{inner_optimal_estimation_MSE} which was the first objective of this section. Theorem \ref{prop_bound_total_mse} rather paves the way to a proof of error bounds on particle filters oriented toward particle filter moment stability. Actually, let us assume that a stricly suboptimal estimator of $X_k$ w.r.t. to the total MSE, denoted by $\widehat{X}_k^{sub}$, is available. For example, $\widehat{X}_k^{sub}$ may come from a Kalman-like filter in a nonlinear case. By optimality of  $\widehat{X}_k^{*}$ and by Theorem \ref{prop_bound_total_mse}  for a sufficiently large $N$, one gets, for $k\geq0$:
    \begin{align}
          e^{tot}_{k,*}\leq e^{tot}_{k,N}< E\left[{\Vert X_k - \widehat{X}_k^{sub}\Vert}^2\right]. \label{filter_comparison}
    \end{align}

   This means that the particle filter has better performance than any other suboptimal filter if the number of particles is sufficiently high. This result is not surprising and observed in practice. However, a rigorous proof of such a result has never been made to the best of our knowledge.  As a result, one can see from equation \eqref{filter_comparison} that is the MSE generated by $\widehat{X}_k^{sub}$ is bounded w.r.t. k then so is $e^{tot}_{k,N}$.  This  seems to be a good alternative in order to show MSE boundedness for a particle filter. In fact, in \cite{reif_stochastic_1999}, under a small error assumption, the stability of the Extended Kalman filter is proven in terms of bounded MSE. Other results of stability of nonlinear filters can be found in \cite{karvonen_stability_2014}. However, this statement is not rigorous for the moment because one still needs a  number of particle increasing with $k$ according to the previous remark. Note that this kind of result is hard to obtain if one considers directly the particle filter because one would typically need some nonlinear stochastic observability condition. See again \cite{karvonen_stability_2014}. Finally, this kind of result is very useful in an output feedback control perspective because it could be a first step toward showing a closed-loop moment stability result of the true state of system, $X_k$, in a nonlinear framework. See \cite{hokayem_stochastic_2012} for an example of outputfeedback moment stability with bounded MSE in a linear context.

\end{remark}
\section{Explicit dual control and Optimal control with an estimation based cost}\label{sec_modified_explicit}
The objective of this section is to use Step \ref{optimal_control_estimation_cost_step} from Section \ref{sec_dp_control+estimation} as a formal justification of a class of dual controllers called Explicit dual controllers. More precisely, a link between classically used empirical losses of information added to the cost and the optimal value of the estimation problem \eqref{sto_control_imperfect_inf_infinite_horizon_state_filter_estimation_cost} is presented.

\subsection{Explicit dual control}

In this section, we focus on Explicit dual controllers based on integrated experiment design. The idea in this case is to add a quantitative measure of the loss of information in the cost of a stochastic optimal control problem like Problem \eqref{sto_opt_control_imperfect_inf_infinite_horizon_state_filter}. If one defines $g^{info}$: $\mathbb{R}^{n_x}\longrightarrow \mathbb{R}^+$ as the loss of information, one can write an infinite-horizon Explicit dual control problem as follows:
\begin{align}
\begin{array}{rrclcc}
\widetilde{V}(\mu)=\displaystyle \underset{\pi_{0}^c,\pi_{0}^e }{\text{min}} & \multicolumn{3}{l}{E^{\pi}\left[\sum_{k=0}^{+\infty}\alpha^k( \tilde{g}^c(\mu_k,U_k)+\tilde{g}^{info}(\mu_k,\widehat{X}_k))| \mu_0 = \mu\right]} \\
\textrm{s.t.} &\mu_{k+1}&=&F\left(\mu_k,Y_{k+1},U_k\right),\\
& U_k & = & \pi_0^c(\mu_k), \\
& \widehat{X}_k & = & \pi_0^e(\mu_k), \; \forall k\geq 0,\\
\end{array}
\label{sto_control+information_imperfect_inf_infinite_horizon_state_filter_detail}
\end{align}
Optimal solutions of Problem \eqref{sto_control+information_imperfect_inf_infinite_horizon_state_filter_detail} exhibit the property of Explicit dual effect. The term 'explicit' comes from the external nature of the coupling between estimation and control realised by $g^{info}$. Note that optimal solutions of Problem \eqref{sto_control+information_imperfect_inf_infinite_horizon_state_filter_detail} also exhibit implicit dual effect for the same reason as those of Problem \eqref{sto_opt_control_imperfect_inf_infinite_horizon_state_filter} do.   
However, Problem \eqref{sto_control+information_imperfect_inf_infinite_horizon_state_filter_detail} is usually destined in practice to be approximated  by an Open-loop finite-horizon problem used inside a dual Explicit Stochastic Model Predictive control scheme. See \cite{mesbah_stochastic_2017} for a review on the subject. In this case, the Open-loop approximation destroys the feedback structure of the policies and the implicit dual effect is lost. The main interest of Problem  \eqref{sto_control+information_imperfect_inf_infinite_horizon_state_filter_detail} is precisely that $g^{info}$ preserves some dual effect even after these two approximations.
Still, the main flaw of this technique is that $g^{info}$ is generally empirical and chosen ad hoc which makes Problem \eqref{sto_control+information_imperfect_inf_infinite_horizon_state_filter_detail} hard to connect with the original control problem  \eqref{sto_opt_control_imperfect_inf_infinite_horizon_state_filter}.

\subsection{Optimal control with an estimation based-cost}
We recall here the modified control problem that appears in Step \ref{optimal_control_estimation_cost_step} from Section \ref{sec_dp_control+estimation}:
 \begin{align}
\begin{array}{rrclcc}
\widetilde{V}(\mu)=\displaystyle \underset{\pi_{0}^c }{\text{min}} & \multicolumn{3}{l}{E^{\pi}\left[\sum_{k=0}^{+\infty}\alpha^k( \tilde{g}^c(\mu_k,U_k)+\tilde{g}^e_*(\mu_k))| \mu_0 = \mu\right]} \\
\textrm{s.t.} &\mu_{k+1}&=&F\left(\mu_k,Y_{k+1},U_k\right),\\
& U_k & = & \pi_0^c(\mu_k), \; \forall k\geq 0.\\
\end{array}
\label{sto_control_imperfect_inf_infinite_horizon_state_filter_estimation_cost_2}
\end{align}
Notice that, since $\tilde{g}^e_*$ is the minimum estimation error given the current distribution $\mu_k$, one could be tempted to choose controls that minimise it w.r.t $\mu_k$ in order to get more information on the system. In this sense, $\tilde{g}^e_*$ can be seen as a measure of a loss of information. Although, it is not a very practical one because it  depends on the value of each past and presents observations and is therefore hard to predict. We would like to use $g^{info}$ as a simpler and a priori measure of information. 

We would like to introduce a new perspective linking the Explicit dual control problem \eqref{sto_control+information_imperfect_inf_infinite_horizon_state_filter_detail} and the the estimation-based control problem \eqref{sto_control_imperfect_inf_infinite_horizon_state_filter_estimation_cost_2} which is itself a step in the resolution of the coupled estimation and control scheme from Section \ref{sec_dp_control+estimation}. We propose to consider that when one solves Problem 
 \eqref{sto_control+information_imperfect_inf_infinite_horizon_state_filter_detail}, one solves a sort of approximation of Problem \eqref{sto_control_imperfect_inf_infinite_horizon_state_filter_estimation_cost_2} and not a modified version of the classical problem \eqref{sto_opt_control_imperfect_inf_infinite_horizon_state_filter}. With this point of view, we try to narrow the gap created by the empirical consideration in  Problem \eqref{sto_control+information_imperfect_inf_infinite_horizon_state_filter_detail}. As a result, one could imagine new explicit dual schemes with better approximations of  $\tilde{g}^e_*$ than current ones.

\section{An application in Aerospace engineering: Terrain-aided navigation}\label{sec_example}

The goal of this section is to give a typical example of application where the practical resolution of the estimation and control steps from Section \eqref{sec_dp_control+estimation} is relevant. We also show that the modelling  assumptions form Section \ref{sec_particle_near_opti} are satisfied in this example. The application under consideration is the problem of localisation and guidance of a drone by Terrain-aided Navigation (TAN).  The objective is to be able to localise a drone and guide it in a 3D space using speed measurements and one dimensional measurement related to the position. In the Cartesian coordinates, we assume that the dynamics of the drone are described as follows:

\begin{itemize}
    \item the state is composed of a 3 dimensional position and a 3 dimensional speed: $X_k={(x_{1,k},x_{2,k},x_{3,k},v_{1,k},v_{2,k},v_{3,k})}$ and the control of a 3 dimensional acceleration $U_k={(u_{1,k},u_{2,k},u_{3,k})}$. Note that $(x_{1,k},x_{2,k})$ represents the horizontal position and $x_{3,k}$ the altitude.
    \item its dynamics (\ref{sto_sys_lin_sat}) is linear with bounded controls, for $k\in \mathbb{N}$,
    \begin{align}
        X_{k+1}=AX_k+BU_k+\xi_k,\label{sto_sys_lin_sat}\qquad
        \Vert U_k \Vert\leq U_{max},
    \end{align}
   where $U_{max}>0$, $A\in \mathbb{R}^{n_x\times n_x}$ and $B\in \mathbb{R}^{n_x\times n_u}$ correspond to a discrete-time second order and $\xi_k\sim \mathcal{N}(0,Q)$ with Q positive semi-definite.
\end{itemize}

We assume that the dynamics has a relatively simple form because the main difficulty of this application is the nature of the observations. Indeed, the only information on the position is a measurement of the difference between the altitude of the drone, $x_{3,k}$, and the altitude of the corresponding vertical point on the ground. We also suppose that the ground is represented by a bounded map, $h_M: \mathbb{R}^2 \longrightarrow \mathbb{R}^+$. In practice, $h_{M}$ is often determined by a smooth interpolation of data points which makes it very nonlinear. Therefore, the observation equation reads:
\begin{align}
    Y_k=h_{det}(X_k)+\eta_k, \label{obs_eq_extended_real_map}
\end{align}
where $h_{det}(X_k) =\begin{bmatrix} v_{1,k}\\v_{2,k}\\v_{3,k}\\ x_{3,k}-{h}_M(x_{1,k},x_{2,k})\end{bmatrix}$ and $\eta_k$ is measurement noise. Its  distribution is assumed to have a density w.r.t the Lebesgue measure denoted by $\rho_{\eta}$ which is also bounded with a bounded support. Note that assuming a bounded sensor noise is relevant in many case and especially in Aerospace engineering.



Intuitively, the use of particle filters is justified in this case by the map $h_M$ which is nonlinear and may have ambiguities resulting in a multi-modal distribution $\mu_k$. Actually, its modes are closely related to the level sets of  $h_M$. As a matter of fact, Kalman filters cannot accurately deal with this problem. 
Moreover, it appears very naturally that dual control is required in this application. Indeed, the quality of the observations depends on the area that is flied over by the drone. If the drone flies over a flat area with constant altitude, then, one measurement of height matches a whole horizontal area and the estimation error on $(x_{1,k},x_{2,k})$ is of the order of magnitude of the size of the area, which can be very large. On the contrary, if the drone flies over a rough terrain, then one measurement of height corresponds to a smaller area on the ground and the estimation error is reduced. Thus, in TAN, information probing consists in flying over informative areas of the ground. Therefore, information measures based on some norm of the gradient of the $h_m$, like the Fisher Information Matrix \cite{tichavsky_posterior_1998}, are relevant. 
Finally, we would like to show that the system \eqref{sto_sys_lin_sat}-\eqref{obs_eq_extended_real_map} is an example of system where Assumptions \ref{as:finite}, \ref{as:bounded} and \ref{as:bounded_2} hold.  This results is summed in Proposition \ref{prop_ex_as}.
\begin{proposition}\label{prop_ex_as}
The system \eqref{sto_sys_lin_sat}-\eqref{obs_eq_extended_real_map} satisfies Assumptions \ref{as:finite}, \ref{as:bounded} and \ref{as:bounded_2}. 
\end{proposition}
\begin{proof}
Clearly, from equation \eqref{obs_eq_extended_real_map}, Assumption \ref{as:finite} is satisfied. In the sequel, we prove that Assumption \ref{as:bounded_2} holds which implies that Assumption \ref{as:bounded} holds too. By independence of $\eta_k$ w.r.t. $X_k$, the likelihood function can be written as follows:
\begin{align*}
    \rho(Y_k,X_K)=\rho_{\eta}(Y_k-h_{det}(X_k))
\end{align*}
From equation \eqref{obs_eq_extended_real_map} and as $\rho_{\eta}$ has a bounded support one can get that
\begin{align}
{\Vert X_k\Vert}^2 \rho(Y_k,X_K)=0\;\textrm{for}\; \Vert (Y_k,X_k)\Vert\;\textrm{sufficiently large} \label{rho_bounded_support}
\end{align}
One can deduce, using the notation in Assumption \ref{as:bounded_2}, that: 
\begin{align*}
     \Vert \rho\phi_1^2\Vert_{\infty}&=\underset{x,y}{\textrm{sup}}\; \vert x_i^2\rho(y,x)\vert<+\infty,\; \mathrm{for}\; i=1,2.\\
\end{align*}
Besides, $\Vert \rho\Vert_{\infty}<+\infty$ by assumption on $\rho_{\eta}$. Finally, the  noise in the dynamics \eqref{sto_sys_lin_sat} being Gaussian,  $\Vert K \Vert_{\infty}<+\infty$ and Assumption \ref{as:bounded_2} holds.

\end{proof}

\section{Conclusion}

In this paper, a general formalisation of the joint problem of nonlinear optimal filtering and discrete-time stochastic optimal control is proposed. 
    Under natural assumptions on the cost function one can justify the use of two steps in the resolution of the problem. The first step is to solve a classical optimal estimation problem. Near-optimality of the empirical mean of a modified particle filter w.r.t. the mean square error has been shown which justifies the use of particle filtering in the case of MSE minimisation. The second step is to solve a modified optimal control problem with a new term coming from  optimal estimation. This establishes a connection with Explicit dual control where a new term representing a measure of information is empirically added to the cost. Actually, this empirical term can be seen as an approximation of the term coming from optimal estimation. Finally, this framework is illustrated by an example coming from Aerospace engineering namely Terrain-Aided Navigation.

\appendix


 

\section{Proof of Lemma \ref{prop_upper_bound_cond_mse}}\label{appendix_1}

 We recall that under Assumption \ref{as:gamma_1},
    $\langle\mu_{k\vert k-1},\rho \rangle>\gamma_k>\frac{\gamma_k}{2}$. Besides, $\forall k\geq1$, for $0<\epsilon_k<1$,  for a.a. $i_k$, there exists $N_k(i_k)>0$ such that, $\forall N\geq N_k(i_k)$:
\begin{align*}
    {P}( \langle\mu_{k\vert k-1}^N,\rho \rangle>\frac{\gamma_k}{2}\vert I_k=i_k)&\geq 1-\epsilon_k.
\end{align*}

Therefore, under Assumptions \ref{as:finite} and \ref{as:bounded}, we consider theorem $3.1$ of \cite{hu_general_2011} with $\psi=\phi_j$, with $p=r=2$ and with $\frac{\gamma_k}{2}$ instead of $\gamma_k$. It implies that, for $ j =1,..,n_x$ and $k\geq1$, for almost all $i_k$, there exist $ C_{k,j}>0$, $ M_{k,j}>0$ and $N_k(i_k)>0$ such that $\forall N\geq N_k(i_k)$: 
\begin{align}
E\left[{\vert \langle\mu_k,\phi_j \rangle - \langle\mu_k^N,\phi_j \rangle\vert}^2\vert I_k=i_k\right]&\leq \frac{C_{k,j}{\Vert\phi_j\Vert}_{k,2}^2}{N},\notag\\  
    E\left[{\Vert \langle\mu_k,\mathrm{Id} \rangle - \langle\mu_k^N,\mathrm{Id} \rangle\Vert}^2\vert I_k=i_k\right]&\leq \frac{\sum_{j=1}^{n}C_{k,j}{\Vert\phi_j\Vert}_{k,2}^2}{N}.\label{bound_error_particle_filter_identity}
\end{align}

By putting end to end the computations in \cite{hu_general_2011}, one can show that, for $k\geq 1$, and $j =1,..,n_x$,  $C_{k,j}$ and $M_{k,j}$ follow the coupled equations \eqref{M_ini} to \eqref{beta_current}.
 From the computation of $C_{k,j}$ and Assumption \ref{as:gamma_1}, one can also show that $N_k(i_k)$ can be chosen as such that $N_k(i_k)\geq\frac{{\Vert\rho\Vert}_{k,2}^2{\Vert K\Vert}^2\underset{j}{\max}\; C_{k-1\vert k-1,j} }{{\vert\frac{\gamma_k}{2}-\langle\mu_{k\vert k-1},\rho \rangle\vert}^2 \epsilon_k}$, with $0<\epsilon_k<1$.

One gets the result by combining equations \eqref{upper_bound_empirical_MSE} and \eqref{bound_error_particle_filter_identity} to obtain  $\forall \epsilon>0$, for almost all $i_k$, $\forall N\geq N_k(i_k)$: 
\begin{align*}
e^{cond}_{k,N}(i_k)&\leq(1+\epsilon)e^{cond}_{k,*}(i_k)+\left(1+\frac{1}{\epsilon}\right)\frac{\sum_{j=1}^{n}C_{k,j}{\Vert\phi_j\Vert}_{k,2}^2}{N},
\end{align*}

  \section{Proof of Theorem \ref{prop_convergence_cond_MSE}}\label{appendix_2}

 As Lemma \ref{prop_optimality_MSE} requires no assumptions,  Lemma \ref{prop_optimality_MSE} and \ref{prop_upper_bound_cond_mse} hold under Assumptions \ref{as:gamma_1}, \ref{as:finite} and \ref{as:bounded} and one gets directly that  $\forall\epsilon>0$, for a.a. $i_k$, for any $ N\geq N_k(i_k)$:
 \begin{align*}
e^{cond}_{k,*}(i_k)\leq e^{cond}_{k,N}(i_k)&\leq(1+\epsilon)e^{cond}_{k,*}(i_k)+\left(1+\frac{1}{\epsilon}\right)\frac{\sum_{j=1}^{n}C_{k,j}{\Vert\phi_j\Vert}_{k,2}^2}{N}.
\end{align*}
 By choosing $\epsilon=\frac{1}{N^q}$ with $0<q<1$, one can obtain from (\ref{bound_error_1})  that, for a.a. $i_k$, $\forall N\geq N_k(i_k)$:

 \begin{align}
 e^{cond}_{k,*}(i_k)\leq e^{cond}_{k,N}(i_k)\leq&\left(1+\frac{1}{N^q}\right)e^{cond}_{k,*}(i_k)+\left(1+N^q\right)\frac{\sum_{j=1}^{n}C_{k,j}{\Vert\phi_j\Vert}_{k,2}^2}{N}.\label{bound_error_1}
 \intertext{Moreover, the right-hand side converges such that:}
 \left(1+\frac{1}{N^q}\right)e^{cond}_{k,*}(i_k)+&\left(1+N^q\right)\frac{\sum_{j=1}^{n}C_{k,j}{\Vert\phi_j\Vert}_{k,2}^2}{N}\underset{N\rightarrow+\infty}{\longrightarrow}e^{cond}_{k,*}(i_k)\label{limit_bound_error_1}.
 \end{align}
 Thus, it is now clear from (\ref{bound_error_1}) and  (\ref{limit_bound_error_1}) that for a.a. $i_k$:
 \begin{align*}
 e^{cond}_{k,N}(i_k)&\underset{N\rightarrow+\infty}{\longrightarrow}e^{cond}_{k,*}(i_k),\notag
 \end{align*}

\section{Proof of Lemma \ref{prop_bound}}\label{appendix_3}

  One defines $C_{k,j}'$, and $M_{k,j}'$ recursively as follows, $\forall k\geq0$, $\forall j=1,..,n$:
  
   \begin{align}
     M_{0,j}'&=3,\label{M_ini'}\\
     C_{0,j}'&=8\widetilde{C},\label{C_ini'}\\
     M_{k,j}'&=2+\alpha_{k,j}'\left(1+\left(\frac{4-\epsilon_k}{1-\epsilon_k}+1\right)M_{k-1,j}'\right)\label{M_current'},\\
     {(C_{k,j}')}^{\frac{1}{2}}&=2^{\frac{3}{2}}{(\widetilde{C})}^{\frac{1}{2}}{(M_{k,j}')}^{\frac{1}{2}}+\frac{2^{\frac{3}{2}}{(\widetilde{C})}^{\frac{1}{2}}\beta_{k,j}'}{{(1-\epsilon_k)}^{\frac{1}{2}}}{(M_{k-1,j}')}^{\frac{1}{2}}\label{C_current'}\\
     &+\frac{{\Vert K\Vert}_{\infty}^{\frac{3}{2}}{\Vert\rho\Vert}_{\infty}\beta_{k,j}'}{(1-\epsilon_k)\frac{\gamma_k}{2}}{(M_{k-1,j}')}^{\frac{1}{2}}{(C_{k-1,j}')}^{\frac{1}{2}}+\Vert K \Vert_{\infty}\beta_{k,j}'{(C_{k-1,j}')}^{\frac{1}{2}}\notag,
\end{align}
     
\begin{align}
     \alpha_{k,j}'={\Vert K\Vert}_{\infty}^2\frac{ {\Vert \rho\Vert}_{\infty}( {\Vert \phi_j^2\rho\Vert}_{\infty}+\frac{\gamma_k}{2})}{\frac{\gamma_k^2}{2}},\qquad
     \beta_{k,j}'=\frac{ {\Vert \rho\Vert}_{\infty}( {\Vert \phi_j\rho\Vert}_{\infty}+\frac{\gamma_k}{2})}{\frac{\gamma_k^2}{2}},\label{beta_current'}.
 \end{align}
  Because of Assumption \ref{as:bounded_2}, the following inequalities hold, $\forall j=1,..,n$:
  \begin{align}
      {\Vert \rho\Vert}&\leq{\Vert \rho\Vert}_{\infty}<+\infty,&
      {\Vert \rho\phi_j^2\Vert}&\leq{\Vert \rho\phi_j^2\Vert}_{\infty}<+\infty,& \label{bound_proof}
      {\Vert \rho\phi_j\Vert}&\leq{\Vert \rho\phi_j\Vert}_{\infty}<+\infty.
  \end{align}
 Thus, by recursion on $k$, $\forall k\geq0$, $\forall j=1,..,n$
 \begin{align*}
      C_{k,j}'&<+\infty,&
    M_{k,j}'&<+\infty.
 \end{align*}

  From the definition of $\alpha_{k,j}$ and $\beta_{k,j}$ in equation \eqref{beta_current},  One needs to be able to bound the term $ \frac{1}{\vert\langle\mu_{k\vert k-1},\rho \rangle-\frac{\gamma_k}{2}\vert}$ from above uniformly in $i_k$.
 To do so, from Assumption \ref{as:gamma_1}, for a.a. $i_k$, one gets:
 \begin{align}
     \langle\mu_{k\vert k-1},\rho \rangle&\geq\gamma_k\geq\frac{\gamma_k}{2},\label{gamma_proof_1}\\
     \langle\mu_{k\vert k-1},\rho \rangle-\frac{\gamma_k}{2}&\geq\frac{\gamma_k}{2}>0,\notag\\
     \frac{1}{\vert\langle\mu_{k\vert k-1},\rho \rangle-\frac{\gamma_k}{2}\vert}=
     \frac{1}{\langle\mu_{k\vert k-1},\rho \rangle-\frac{\gamma_k}{2}}&\leq\frac{1}{\frac{\gamma_k}{2}}.\label{gamma_proof_2}
 \end{align}

 Consequently, from (\ref{bound_proof}) and (\ref{gamma_proof_2}), $\forall k\geq0$, $\forall j=1,..,n_x$, for a.a. $i_k$:
 \begin{align}
      \alpha_{k,j}&\leq \alpha_{k,j}',\label{alpha_beta_proof}&
      \beta_{k,j}&\leq \beta_{k,j}'.
 \end{align}
 
 Finally, from equations (\ref{gamma_proof_1}), (\ref{alpha_beta_proof}),  equations (\ref{M_ini}) to (\ref{beta_current}) and equations (\ref{M_ini'}) to (\ref{beta_current'}),  one can show by recursion on $k$, that $\forall k\geq0$, $\forall j=1,..,n$, for a.a. $i_k$:
 \begin{align*}
      C_{k,j}&\leq C_{k,j}',&
        M_{k,j}&\leq M_{k,j}'.
 \end{align*}
 Besides, since $\gamma_k$, ${\Vert \rho\Vert}_{\infty}$, ${\Vert \rho\phi_j^2\Vert}_{\infty}$ and ${\Vert \rho\phi_j\Vert}_{\infty}$ do not depend on $i_k$, $C_{k,j}'$ and $M_{k,j}'$  do not depend on $i_k$ either and one gets the result.

\section{Proof of Theorem \ref{prop_bound_total_mse}}\label{appendix_4}

  Under Assumptions \ref{as:gamma_1}, \ref{as:finite} and \ref{as:bounded_2}, Lemma \ref{prop_upper_bound_cond_mse} holds and implies that,  for a.a. $i_k$, $\forall N\geq N_k(i_k)$, $\forall \epsilon>0$:
\begin{align}
     E\left[{\Vert X_k - \widehat{X}_k^N\Vert}^2\vert I_k=i_k\right]&\leq\left(1+\epsilon\right)E\left[{\Vert X_k - \widehat{X}_k^*\Vert}^2\vert I_k=i_k\right]+\left(1+\frac{1}{\epsilon}\right)\frac{\sum_{j=1}^{n}C_{k,j}{\Vert\phi_j\Vert}_{k,2}^2}{N},\label{bound_error_proof}
\end{align}
where $   N_k(i_k)=\frac{{\Vert\rho\Vert}_{k,2}^2{\Vert K\Vert}^2\underset{j}{\text{max}}\; C_{k-1,j} }{{\vert\frac{\gamma_k}{2}-\langle\mu_{k\vert k-1},\rho \rangle\vert}^2 \epsilon_k}.$

First, one can notice from (\ref{bound_proof}), (\ref{gamma_proof_2}) and Lemma \ref{prop_bound} that, $\forall k\geq1$,  for a.a. $i_k$:
\begin{align*}
    N_k(i_k)\leq\frac{{\Vert\rho\Vert}_{\infty}^2{\Vert K\Vert}^2_{\infty} }{{(\frac{\gamma_k}{2})}^2 \epsilon_k}\underset{j}{\text{max}}\; C_{k-1,j}'\equiv \bar{N}_k.
\end{align*}
Note that $\bar{N}_k$ does not depend on $i_k$ then (\ref{bound_error_proof}) is true for a number of particles independent of $i_k$ i.e 
 $\forall \epsilon>0$, $\forall N\geq \bar{N}_k$, for a.a. $i_k$, :
\begin{align*}
     E\left[{\Vert X_k - \widehat{X}_k^N\Vert}^2\vert I_k=i_k\right]&\leq\left(1+\epsilon\right)E\left[{\Vert X_k - \widehat{X}_k^*\Vert}^2\vert I_k=i_k\right]+\left(1+\frac{1}{\epsilon}\right)\frac{\sum_{j=1}^{n}C_{k,j}{\Vert\phi_j\Vert}_{k,2}^2}{N}.
\end{align*}
By using Lemma \ref{prop_bound} again:
\begin{align}
     E\left[{\Vert X_k - \widehat{X}_k^N\Vert}^2\vert I_k=i_k\right]&\leq\left(1+\epsilon\right)E\left[{\Vert X_k - \widehat{X}_k^*\Vert}^2\vert I_k=i_k\right]+\left(1+\frac{1}{\epsilon}\right)\frac{\sum_{j=1}^{n}C_{k,j}'{\Vert\phi_j\Vert}_{k,2}^2}{N}.\notag\\
     \intertext{Now one is able to integrate over $i_k$, which leads to:}
     E\left[{\Vert X_k - \widehat{X}_k^N\Vert}^2\right]&\leq\left(1+\epsilon\right)E\left[{\Vert X_k - \widehat{X}_k^*\Vert}^2\right]+\left(1+\frac{1}{\epsilon}\right)\frac{\sum_{j=1}^{n}C_{k,j}'E\left[{\Vert\phi_j\Vert}_{k,2}^2\right]}{N}.\label{bound_error_proof_2}
\end{align}

We recall that  $E\left[{\Vert X_k - \widehat{X}_k^*\Vert}^2\vert I_k=i_k\right]$ and ${\Vert\phi_j\Vert}_{k,2}^2$ are integrable w.r.t. $i_k$ because $X_k$ is square-integrable by Assumption \ref{as:square_int}. Moreover, from Lemma \ref{prop_optimality_MSE}, and by taking $\epsilon=\frac{1}{N^q}$ with $0<q<1$, one can get additionally that $\forall k\geq0$, $\forall N\geq \widebar{N}_k$:
\begin{align*}
    e^{tot}_{k,*}\leq e^{tot}_{k,N}\leq\left(1+\frac{1}{N^q}\right)e^{tot}_{k,*}+&\left(1+N^q\right)\frac{\sum_{j=1}^{n}C_{k,j}'E\left[{\Vert\phi_j\Vert}_{k,2}^2\right]}{N}<+\infty.
\end{align*}
 The convergence result is straightforward from the previous equation.

\bibliographystyle{bibft}
\bibliography{bibfile}

\end{document}